\documentclass[a4paper,10pt]{article}
\usepackage{amsmath,amssymb,amsfonts,amsthm,mathrsfs}
\usepackage{hyperref}
\usepackage{xcolor}
\usepackage{tikz-cd}
\usetikzlibrary{decorations.pathmorphing} 
\usepackage[all]{xy}
\usepackage[]{algorithm2e}
\usepackage{enumitem}
\usepackage{setspace}

\usepackage{import}
\graphicspath{{Figures/}}
\usepackage[margin=20pt]{caption} 

 \usepackage[margin=1.5 in]{geometry}
  \usepackage[style=alphabetic,sorting=nyt,maxbibnames=7, maxcitenames=5, backend=biber, style=alphabetic, doi=false, url=false]{biblatex}
\newtheorem{thm}{Theorem}[section]

\newtheorem{lem}[thm]{Lemma}
\newtheorem{prop}[thm]{Proposition}
\newtheorem{thma}{Theorem}

\newtheorem*{claim*}{Claim}

\theoremstyle{definition}
\newtheorem{dfn}[thm]{Definition}
\newtheorem{rem}[thm]{Remark}

\renewcommand{\tilde}{\widetilde}

\newcommand{\RR}{\mathbb{R}}
\newcommand{\NN}{\mathbb{N}}
\newcommand{\ZZ}{\mathbb{Z}}

\newcommand{\E}{\operatorname{E}} 
\newcommand{\F}{\operatorname{F}} 

\newcommand{\Ftp}[1]{\mathrm{F}_{#1}}
\newcommand{\Flg}[1]{\mathrm{FP}_{#1}}

\newcommand{\Hom}{\operatorname{Hom}}

\newcommand{\id}{\mathrm{id}}

\newcommand{\VR}[2]{\operatorname{VR}_{#1}({#2})}
\newcommand{\diff}[1]{\partial_{#1}} 
\newcommand{\chains}[3]{\operatorname{C}_{#1}(\VR{#2}{#3})} 
\newcommand{\homology}[3]{\operatorname{H}_{#1}(\VR{#2}{#3})} 
\newcommand{\chain}{\operatorname{C}}
\newcommand{\hlgy}{\operatorname{H}}

\DeclareMathOperator{\im}{im}

\DeclareMathOperator{\Map}{Map}

\DeclareMathOperator{\Ind}{Ind}

\title{Computer-assisted methods in $\Sigma$-theory}

\author{Elisa Hartmann}

\begin{document}

\maketitle

\begin{abstract}
 We develop a procedure to recognize whether a character belongs to $\Sigma^m$. In order to apply it, we just need to know that the ambient group is of type $\mathrm{FP}_m$ or of type $\mathrm{F}_2$ and that the word problem is solvable for this group. Then finite data are sufficient proof of membership in $\Sigma^m$, not just for the given character but also for a neighborhood of it.
\end{abstract}

\begin{spacing}{1.1}
\section{Introduction}

The two properties \emph{type $\Ftp m$} and \emph{type $\Flg m$} are well-known geometric invariants on groups. Type $\Ftp m$ is often considered the homotopical version, while type $\Flg m$ is considered the homological version. Associated with them are directional versions $\Sigma^m(G)$ for homotopical and $\Sigma^m(G;\ZZ)$ for homological. They are often called \emph{Sigma-invariants} or \emph{BNSR-invariants} named after their inventors Bieri--Neumann--Strebel--Renz \cite{Bieri1988,Renz1988,Strebel2012}. We propose procedures that can be implemented on a computer to compute $\Sigma$-invariants.

The geometric intuition comes from metric spaces. That is, if a group $G$ is finitely generated by a subset $X$, then there is an assignment
\begin{align*}
 \ell:G&\to \ZZ_{\ge 0}\\
 g&\mapsto \min\{n\mid g=_Gx_1^{\pm}\cdots x_n^{\pm};x_i\in X\}
\end{align*}
called the \emph{word length according to $X$}. We can assign a left-invariant metric on $G$ by defining $d(g,h):=\ell(g^{-1}h)$. This metric, of course, depends on the choice of generating set $X$, but its geometry at infinity does not. So, if $T$ is a metric space, then $\E T$ denotes the free simplicial set on $T$. That is, its $q$-simplices are $\E T_q:=T^{q+1}$, $q=0,1,\ldots$ where
\begin{align*}
 d_i:\E T_q&\to \E T_{q-1}\\
 (t_0,\ldots,t_q)&\mapsto (t_0,\ldots,\hat t_i,\ldots,t_q)
\end{align*}
stands for the $i$th face map $i=0,\ldots,q$ and
\begin{align*}
 s_i:\E T_q&\to \E T_{q+1}\\
 (t_0,\ldots,t_q)&\mapsto (t_0,\ldots,t_i,t_i,\ldots,t_q)
\end{align*}
stands for the $i$th degeneracy map $i=0,\ldots,q$. Then for each $k\ge 0,q\in 
\ZZ_{\ge 0}$ we have a collection of $k$-small $q$-simplices
\[
 \Delta^q_k:=\{(t_0,\ldots,t_q)\mid d(t_i,t_j)\le k;i,j=0,\ldots,q\}
\]
The subcomplex $\VR k T$ of $\E T$ with $q$-simplices ${\VR k T }_q:=\Delta^q_k$ is also called the \emph{Vietoris--Rips complex for the value $k$}. We emphasize that the usual definition of the Vietoris--Rips complex does not impose an ordering on the vertices that determine a simplex \cite{Hausmann1995}. We are using an oriented version of the Vietoris--Rips complex, that is, a Vietoris--Rips simplicial set. If we let $k$ vary to $\infty$, then $(\VR k T )_{k\ge 0}$ form a filtration of $\E T$. For the homological part of this paper, simplicial homology will play an important role. If $k\ge 0$ and $G$ is a countable group (we refer to Section~\ref{sec:def} for the metric on a countable $G$ in the absence of a finite generating set), then $(\chain_q(\VR k G),\partial_q)$ denotes the simplicial chain complex assigned to the space $\VR k G$. That is,
\[
 \chains q k G=\ZZ[\Delta_k^q]
\]
and
\begin{align*}
 \diff q:\chains q k G&\to \chains {q-1} k G\\
 (g_0,\ldots,g_q)&\mapsto \sum_{i=0}^q(-1)^i(g_0,\ldots,\hat g_i,\ldots,g_q).
\end{align*}
Coincidentally, $\ZZ[\Delta^q_k]$ is also a free $G$-module and $(\ZZ[\Delta^q_k],\partial_q)$ form a filtration of the standard resolution $(\ZZ[G^{q+1}],\partial_q)=(\chain_q(\E G),\partial_q)$ of the constant $G$-module $\ZZ$. We find it convenient to think of $\chain_*(\VR k G)$ as the graded $\ZZ[G]$-module or $\ZZ$-module $\bigoplus_{q\ge0}\chain_q(\VR k G)$ and write $\chain_*(\VR k G)^{(m)}$ for the \emph{$m$-skeleton} 
\[
\chain_*(\VR k G)^{(m)}:=\bigoplus_{q=0}^m\chain_q(\VR k G).
\]
We are now ready to state two of our main results:
\begin{thma}
\label{thma:mulg}
  A countable group $G$ is of type $\Flg m$ if and only if there exist $K_0\ge 1$ and for every $k\ge K_0$ a chain homomorphism of $\ZZ G$-complexes 
  \[
  \mu_*:\chain_*(\VR k G)^{(m)}\to \chain_*(\VR {K_0} G)^{(m)}
  \]
  extending the identity on $\ZZ$.
 \end{thma}
A proof of Theorem~\ref{thma:mulg} can be found in Theorem~\ref{thm:chainendo-fpm}. At this point, we note that $m$ needs to be finite. The result of Theorem~\ref{thma:mulg} (and the same result for type $\mathrm F_m$) does not necessarily apply to type $\mathrm{FP}_\infty$. As $m$ increases, $K_0$ (depending on $m$) also increases.
 
 If $\chi:G\to \RR$ is a character on a group, then the map $v:\ZZ[\Delta^q_k]\to \RR$ is defined on simplices $\sigma=(g_0,\ldots,g_q)\in \Delta^q_k$ by $v(\sigma)=\min_{0\le i\le q}\chi(g_i)$. Then $v$ maps a chain $\sum_{\sigma\in\Delta^q_k}n_\sigma \sigma$ to $\min_{\sigma: n_\sigma\not=0}v(\sigma)$. The mapping $v$ specifies a valuation on $\chains q k G$ extending $\chi$ \cite{Bieri1988}.
\begin{thma}
\label{thma:varphilg}
 If $G$ is a group of type $\Flg m$ and $\chi:G\to \RR$ a non-zero character, then $\chi\in \Sigma^m(G;\ZZ)$ if and only if for every $k\ge 0$ there exists a chain endomorphism of $\ZZ G$-complexes
 \[
 \varphi_*:{\chains * k G}^{(m)}\to {\chains * k G}^{(m)}
 \]
 extending the identity on $\ZZ$ with $v(\varphi_q(c))-v(c)>0$ for every $c\in \chains q k G$ and $q=0,\ldots,m$.
\end{thma}
A proof of Theorem~\ref{thma:varphilg} can be found in Theorem~\ref{thm:sigmam-crit}. 

In order to compute the homological $\Sigma$-invariants of a group $G$ we first need to show that it is of type $\Flg m$. Then finding a chain endomorphism $\varphi$ extending the identity on $\ZZ$ that respects the filtration and raises the $\chi$-value proves that $\chi$ belongs to $\Sigma^m(G;\ZZ)$. We can also find a chain endomorphism $\mu$ that extends the identity on $\ZZ$ that maps $\ZZ[G^{q+1}]$ to $\chains q {K_0} G$ for some $K_0$. The result is Procedure~\ref{algo:sigmam} (see Section~\ref{sec:crit1}) with input a group $G$ of type $\Flg m$, some finite data assigned to the group and a character $\chi$ on $G$. Only a finite amount of data needs to be guessed to prove that $\chi\in \Sigma^m(G;\ZZ)$. If, on the other hand, $\chi\not\in \Sigma^m(G;\ZZ)$, then the output is always ``maybe''. Our new $\Sigma^m$-criterion (Procedure~\ref{algo:sigmam}) generalizes the well-known $\Sigma^1$-criterion \cite{Strebel2012}. It is also very similar to the $\Sigma^m$-criterion \cite[Theorem~4.1]{Bieri1988}: The setting for this criterion is a free resolution
\[
 \cdots\to F_m\to F_{m-1}\to \cdots\to F_0\to \ZZ 
\]
of the constant $G$-module $\ZZ$ with $F_m,\cdots,F_0$ finitely generated. Then 
$\chi\in \Sigma^m(G,\ZZ)$ if and only if there exists a chain endomorphism 
$\varphi:F_q\to F_q$ extending the identity on $\ZZ$ with $v(\varphi(x))-v(x)>0$ 
for every basis element $x\in F_q$ and $q=0,\ldots,m$. This results in another 
algorithm that computes $\Sigma$-invariants since $\varphi$ only needs to be 
assigned values on finitely many basis elements. Its input is a group of type 
$\Flg m$, a free resolution of $\ZZ$ with finite $m$-skeleton (which exists, 
since $G$ is of type $\Flg m$ but might be hard to compute) and a character 
$\chi$ on $G$. We hope a homological connecting vector (see 
Section~\ref{sec:crit1} for a definition) is easier to obtain than a free 
resolution of $\ZZ$ with finite $m$-skeleton, whence we think 
Procedure~\ref{algo:sigmam} is the more promising choice. Also, if we have found 
a chain endomorphism $\varphi$ that is a witness for a non-zero character 
$\chi$ to belong to $\Sigma^m$, then it is a witness for characters in a 
neighborhood $U(\varphi)$ of $\chi$ in $\Hom(G,\RR)$.

Now we discuss the claim in the abstract for the homological case.
\begin{thma}
\label{thma:algo1}
 Let $G$ be a group of type $\mathrm{FP}_m$ with solvable word-problem. There exists $k\ge1$ such that for every non-zero character $\chi:G\to \RR$ the following are equivalent:
 \begin{enumerate}
  \item $\chi\in\Sigma^m(G;\ZZ)$;
  \item there exists a chain endomorphism of $\ZZ G$-complexes 
  \[
  \varphi_*:{\chains * k G}^{(m)}\to {\chains * k G}^{(m)}
  \]
  extending the identity on $\ZZ$ with $v(\varphi_q(c))-v(c)>0$ for every $c\in \chains q k G$ and $q=0,\ldots,m$;
  \item Procedure~$\ref{algo:sigmam}$ terminates with output ``yes''.
 \end{enumerate}
\end{thma}
A proof of Theorem~\ref{thma:algo1} can be found in Theorem~\ref{thm:algo1}. In general, membership in $\Sigma^m(G;\mathbb Z)$ is Turing undecidable \cite{Cavallo2017}. Assume that the word problem in $G$ is Turing decidable. Assume also that $G$ is of type $\mathrm{FP}_m$ and that we know the relevant $k\ge 1$ from Theorem~\ref{thma:algo1}. Then membership in $\Sigma^m(G;\ZZ)$ is Turing semi-decidable.

Now we discuss the homotopical version. If $S,T$ are simplicial sets, then $\Map(S,T)$ denotes the set of simplicial maps $S'\to T$ with $S'$ a refinement of $S$ obtained by a finite number of barycentric refinements. 
\begin{thma}
\label{thma:mutp}
A countable group $G$ is of type $\Ftp m$ if and only if there exists $K_0\ge 1$ and for every $k\ge K_0$ a $G$-equivariant map 
\[
\mu\in \Map(\mathrm{VR}_k(G)^{(m)},{\VR {K_0} G}^{(m)}).
\]
\end{thma}
A proof of Theorem~\ref{thma:mutp} can be found in Theorem~\ref{thm:muhomotopy}.

If $\chi:G\to \RR$ is a character on a group and $\tau:=(g_0,\ldots,g_q)\in G^{q+1}$ a simplex, then 
\[
v(\tau):=\min(\chi(g_0),\ldots,\chi(g_q))
\]
and if $S\subseteq \E G$ is a finite simplicial subset, then 
\[
v(S):=\min_{\tau\in S}v(\tau).
\]
This defines a valuation on $\VR k G$ and $\E G$ extending $\chi$.
\begin{thma}
\label{thma:varphitp}
 If $G$ is a group of type $\Ftp m$ and $\chi:G\to \RR$ a nonzero character, then $\chi\in \Sigma^m(G)$ if and only if for every $k\ge 0$ there exists a $G$-equivariant map 
 \[
 \varphi\in \Map({\VR k G }^{(m)},{\VR k G}^{(m)})
 \]
 with $v(\varphi(\sigma))-v(\sigma)>0$ for every $\sigma\in \Delta^m_k$.
\end{thma}
A proof of Theorem~\ref{thma:varphitp} can be found in Theorem~\ref{thm:crit2}.

 In order to compute the homotopical $\Sigma$-invariants of a group $G$ we first need to know that it is of type $\Ftp m$. The strategy is similar to the homological case. Instead of working with chain endomorphisms extending the identity on $\ZZ$, we work with simplicial maps on an $n$-fold barycentric refinement of the domain. If $G$ is of type $\Ftp m$, then a mapping $\varphi$ that respects the filtration and raises the $\chi$-value proves that $\chi\in \Sigma^m(G)$. In addition, there is a mapping $\mu$ that maps $\E G^{(m)}$ to ${\VR {K_0} G}^{(m)}$ for some $K_0\ge 0$. The resulting procedure (see Procedure~\ref{algo:sigma2} in Section~\ref{sec:crit2}) can be compared with the $\Sigma^2$-criterion \cite{Renz1988}, which can also be implemented by hand. Also in the homotopical case, a mapping $\varphi$ that is a witness for a non-zero character $\chi$ to belong to $\Sigma^m$ is also a witness for characters in a neighborhood $U(\varphi)$ of $\chi$ in $\Hom(G,\RR)$.
 
 The claim in the abstract also holds for homotopical $\Sigma$-invariants.
 \begin{thma}
\label{thma:algo2}
 Let $G$ be a group of type $\mathrm{F}_2$ and solvable word-problem. There exists some $k\ge1$ such that for every non-zero character $\chi:G\to \RR$ the following are equivalent:
 \begin{enumerate}
  \item $\chi\in\Sigma^2(G)$;
  \item there exists a $G$-equivariant mapping 
  \[
  \varphi\in\Map({\VR k G}^{(2)},{\VR k G}^{(2)})
  \]
  with $v(\varphi(\sigma))-v(\sigma)>0$ for every $\sigma\in \Delta^2_k$;
  \item Procedure~$\ref{algo:sigma2}$ terminates with the output ``yes''.
 \end{enumerate}
\end{thma}
A proof of Theorem~\ref{thma:algo2} can be found in Theorem~\ref{thm:algo2}. Given a finite presentation of $G$ the relevant number $k$ in Theorem~\ref{thma:algo2} is easy to obtain. Assume that the word problem in $G$ is Turing decidable. Assume also that we know a finite presentation of $G$. Then membership in $\Sigma^2(G)$ is Turing semi-decidable.
 
 There is a version of the $\Sigma$-theory for locally compact Hausdorff groups \cite{Hartmann2024b,Hartmann2024c}. In this context, chain homomorphisms, respectively, simplicial maps, also serve as computational tools, but they are less easy to obtain and describe than in the abstract groups case.

We hope that Procedure~\ref{algo:sigmam} and Procedure~\ref{algo:sigma2} are a useful new tool for studying $\Sigma$-invariants. Ultimately, if we know that the ambient group is of type $\Flg m/\Ftp 2$ and if the word problem is solvable for this group, then guessing finite data amounts to computing $\Sigma$-invariants. The classes of examples that we have in mind, where we have the type $\Ftp \infty$ and can compute the word problem, are certain types of Artin groups \cite{Bux1999,Meier1998,Varisco2021,Meier1998,Meier1997,Almeida2015a,Almeida2017,Almeida2015b} and hyperbolic groups \cite{Bridson1999,Holt2001}. 

\subsubsection*{Acknowledgements}
Many thanks for helpful conversations go to Kai-Uwe Bux, Ilaria Castellano, Dorian Chanfi, Tobias Hartnick, José Quintanilha, Stefan Witzel, and Xiaolei Wu. The author also thanks the anonymous reviewer for their helpful feedback and Eamonn O'Brien for his help in fixing a few language issues.

This research was funded in part by the DFG grant BU 1224/4-1
 within SPP 2026 Geometry at Infinity.
 
 We also acknowledge the support of the Deutsche Forschungsgemeinschaft (DFG, German Research Foundation) –
Project-ID 491392403 – TRR 358.

\section{Definition}
\label{sec:def}
In what follows, we consider the directed system $\VR k X$ of the $k$-Vietoris--Rips complex of $X$ over $k\to\infty$ as an $\mathrm{Ind}$-object in the category of simplicial sets and study the $\mathrm{Ind}$-homology groups, respectively $\mathrm{Ind}$-homotopy groups, thereof. This motivates the definition of \emph{essentially trivial} at the end of this discussion. The reader not interested in the category-theoretic discussion may skip ahead to this definition.

 If $\mathcal C$ is a category, then $\Ind(\mathcal C)$ describes the category of ind-objects in $\mathcal C$. Objects of this category are functors $X:\mathcal I\to \mathcal C$ where $\mathcal I$ is a small filtered category. We also write $(X_i)_i$ with $X_i:=X(i)$ and $\mathcal I$ is usually a poset. If $(X_i)_i$ and $(Y_j)_j$ are two ind-objects, then the set of morphisms between $(X_i)_i$ and $(Y_j)_j$ is given by
\[
 \Hom_{\Ind(\mathcal C)}((X_i)_i,(Y_j)_j)=\varprojlim_i\varinjlim_j \Hom_{\mathcal C}(X_i,Y_j)
\]
\cite[Chapter~8]{Grothendieck1963}. That is, a morphism between ind-objects $X:\mathcal I\to \mathcal C$ and $Y:\mathcal J\to \mathcal C$ is given by a map $\varepsilon:\mathcal I\to \mathcal J$ and a $\mathcal C$-morphism $\varphi_i:X_i\to Y_{\varepsilon(i)}$ for each $i\in \mathcal I$ such that for each $i\to i'\in \mathcal I$ there exists $j\in\mathcal J$ with $\varepsilon(i')\to j\in \mathcal J$ such that the diagram
\[
\xymatrix{
 X_i\ar[d]\ar[r]^{\varphi(i)}
 & Y_{\varepsilon(i)}\ar[dr]\\
 X_{i'}\ar[r]_{\varphi(i')}
 &Y_{\varepsilon(i')}\ar[r]
 &Y_j
 }
\]
commutes. Two data $(\varepsilon,(\varphi_i)_i),(\delta,(\psi_i)_i)$ define the same morphism if for each $i\in \mathcal I$ there is some $j\in \mathcal J$ with $\varepsilon(i)\to j,\delta(i)\to j\in \mathcal J$ and
\[
\xymatrix{
&Y_{\varepsilon(i)}\ar[rd]&\\
 X_i\ar[ru]^{\varphi_i}\ar[r]_{\psi_i}
 &Y_{\delta(i)}\ar[r]
 &Y_j
 }
\]
commutes \cite{Abels1997}. If $\mathcal C$ admits all finite limits and finite colimits (for example if $\mathcal C$ is the category of abelian groups) and if $\varphi:\mathcal J\to \Ind(\mathcal C)$ is a finite diagram given by $\phi:\mathcal J\times \mathcal I\to \mathcal C$, then
\[
 \varprojlim \varphi\cong (\varprojlim_j \phi(j,i))_i
\]
and
\[
 \varinjlim \varphi \cong (\varinjlim_j \phi(j,i))_i
\]
\cite[Corollaire~8.9.2]{Grothendieck1963}. Also, if $F:\mathcal C\to \mathcal D$ is a functor between categories, then it naturally induces a functor $\Ind(F):\Ind(\mathcal C)\to \Ind(\mathcal D)$ by $\Ind(F)((A_i)_i)=(F(A_i))_i$ and $\Ind(F)(\varepsilon,(\varphi_i)_i)=(\varepsilon,F(\varphi_i)_i)$. If an ind-object $(A_i)_i\in \Ind(\mathcal C)$ is indeed isomorphic to some $B\in \mathcal C$ where $\mathcal C$ is naturally fully faithfully embedded in $\Ind(\mathcal C)$ (via $B\mapsto (B)_{1\in\{1\}}$), then
\[
 \varinjlim_i A_i\cong B.
\]
\cite[Proposition~6.3.1]{Kashiwara2006}. We are mainly interested in a particular filtration $(X_i)_i$ of a contractible space $X$ assigned to a group or, more generally, a metric space. Its homotopy groups have the property
\[
\varinjlim_i \pi_q(X_i)=\begin{cases}
                         \ast & q=0\\
                         1 & q=1\\
                         0 & q\ge 2.
                        \end{cases}
\]
Basically, the homotopical finiteness properties determine when $(\pi_q(X_i))_i$ is isomorphic to $\ast,1,0$ as ind-objects. This is a property of the group or metric space. In the homological setting
\[
 \varinjlim_i \tilde H_q(X_i;\ZZ)=0
\]
for every $q\ge 0$. The homological finiteness properties determine when $(\tilde \hlgy_q(X_i;\ZZ))_i$ is isomorphic to $0$ in $\Ind(\mathrm{Ab})$. Ultimately, this is a metric property of the group or, more generally, of a metric space. A filtration $(A_i)_i$ is \emph{essentially trivial} if it is isomorphic to the final object when regarded as an ind-object. In $\mathrm{Ab}$ this means that for every $i\in I$ there exists $j\in I$ with $A_i\to A_j$ the $0$-homomorphism. We say that $A_i$ \emph{vanishes} in $A_j$. Two filtrations $(X_i)_i$ and $(Y_j)_j$ of the same object are \emph{cofinal} if for every $i$ there exists some $j$ with $X_i\subseteq Y_j$ and for every $j$ there exists some $i$ with $Y_j\subseteq X_i$. It is easy to see that two cofinal filtrations are isomorphic as ind-objects.

In the course of this paper, we regard countable groups and certain subsets of the group as metric spaces. That is, a finitely generated group is endowed with the word-length metric of some generating set. Every subset inherits the subspace metric from the ambient group. Given a countable group, not necessarily finitely generated, we can define a metric on it by embedding it in a finitely generated group -- which is always possible \cite{delaHarpe2000,Baumslag1993}.

If $X$ is a metric space and $c\ge 0$, then a finite sequence $a_0,\ldots, a_n\in X$ of points defines a \emph{$c$-path} if $d(a_i,a_{i+1})\le c$ for every $i=0,\ldots,n-1$. The space $X$ is \emph{$c$-coarsely connected} if every two points can be joined by a $c$-path. It is \emph{coarsely connected} if it is $c$-coarsely connected for some $c\ge 0$\cite[Definition~3.B.1]{Cornulier2016}. A countable group is finitely generated if and only if it is coarsely connected \cite[Proposition~4.B.8]{Cornulier2016}.

\begin{lem}
 A non-empty metric space $X$ is $c$-coarsely connected if and only if $\homology 0 c X=\ZZ$.
\end{lem}
\begin{proof}
 Suppose $X$ is $c$-coarsely connected. For every two vertices $x,y\in \VR c X$ there exists a $c$-path in $X$ that joins $x$ to $y$. This path also describes an edge path in $\VR c X$ joining the vertices $x,y$. Thus, $\VR c X$ is connected. Conversely, suppose that $\VR c X$ is connected. For every two points $x,y\in X$ there exists an edge path in $\VR c X$ joining $x$ to $y$. The intermediate vertices of this path describe a $c$-path in $X$ that joins $x$ to $y$. Thus, $X$ is $c$-coarsely connected. 
 
  Now, $\VR c X$ is connected and non-empty if and only if $\hlgy_0(\VR c X)=\ZZ$. 
\end{proof}

\begin{rem} 
 A metric space $X$ is \emph{coarsely simply connected} if $\VR c X$ is connected for some $c\ge 0$ and for every $c'\ge c$ there exists $c''\ge c'$ such that the homomorphism $\pi_1(\VR{c'}X)\to \pi_1(\VR{c''}X)$ induced by the inclusion $\VR{c'}X\subseteq \VR{c''}X$ is trivial \cite[Proposition~6.C.2]{Cornulier2016}. This property is equivalent to \cite[Definition~6.A.5]{Cornulier2016}.
\end{rem}

\begin{dfn}
\label{dfn:finiteprops-metric}
  A metric space $X$ is \emph{of type} $\Flg m$ if the direct system $(\tilde \hlgy_q(\VR k X))_k$ of reduced simplicial homology groups is essentially trivial for each $q=0,\ldots,m-1$.
  
  The space $X$ is \emph{of type} $\Ftp m$ if the direct system $(\pi_q(\VR k X))_k$ of homotopy groups is essentially trivial for each $q=0,\ldots,m-1$.
\end{dfn}
Compare Definition~\ref{dfn:finiteprops-metric} with \cite[Definition~1.4]{Hartnick2022}. Shortly, we compare Definition~\ref{dfn:finiteprops-metric} with the standard definitions given as follows:

If $m\in \mathbb N$, then a group $G$ is of type $\Flg m$ if there is a projective resolution of $G$-modules
 \[
  \cdots\to P_m\to P_{m-1}\to \cdots \to P_0\to \mathbb Z\to 0
 \]
with $P_0,\ldots,P_m$ finitely generated.

If a group $G$ is of type $\Flg m$ and $\chi:G\to \mathbb R$ is a character on $G$, then $\chi\in \Sigma^m(G;\mathbb Z)$ if there is a projective resolution of $G_\chi$-modules
 \[
  \cdots\to Q_m\to Q_{m-1}\to \cdots \to Q_0\to \mathbb Z\to 0
 \]
with $Q_0,\ldots,Q_m$ finitely generated.

A group $G$ is \emph{of type $\mathrm F_m$}, if there exists a $\mathrm K(G,1)$ with finite $m$-skeleton.

A character $\chi:G\to \mathbb R$ is in $\Sigma^m(G)$ if there is a $K=\mathrm K(G,1)$ with finite $m$-skeleton such that the subcomplex $C_v$ of the universal covering $C$ of $K$ with respect to a valuation $v=v_\chi$ is $(m-1)$-connected (meaning that $\pi_q(C_v)=0$ for $q=0,\ldots,m-1$).

Here, a metric $d$ (assigned to a group $G$) refers to the word metric associated to a choice of finite generating set, the existence of which is type $\mathrm F_1=\mathrm{FP}_1$. A metric $d$ assigned to $G_\chi$ refers to the restriction of $(G,d)$ to $G_\chi:=\{g\in G\mid \chi(g)\ge 0\}$.

\begin{thm}(\cites[Corollary~9]{Alonso1994}[Theorem~7.6]{Hartmann2024b}, 
Theorem~\ref{thm:homotopic_metric})
\label{thm:metric}
 If $G$ is a countable group, then
 \begin{itemize}
  \item $G$ is of type $\Ftp m$ if and only if the metric space $(G,d)$ is of type $\Ftp m$;
  \item $G$ is of type $\Flg m$ if and only if the metric space $(G,d)$ is of type $\Flg m$.
 \end{itemize}
 If $\chi:G\to \mathbb R$ is a character on $G$, then 
 \begin{itemize}
  \item $\chi\in \Sigma^m(G;\mathbb Z)$ if and only if $(G_\chi,d)$ is of type $\Flg m$;
  \item $\chi\in \Sigma^m(G)$ if and only if $(G_\chi,d)$ is of type $\Ftp m$.
 \end{itemize}
\end{thm}
We will use Theorem~\ref{thm:metric} as our working definition for finiteness properties of groups and $\Sigma$-invariants.

It is well known that type $\Ftp m$ and type $\Flg m$ are quasi-isometric invariants of a group \cite[Corollary~9]{Alonso1994}. Indeed, they are also coarse invariants: A mapping $\varphi:X\to Y$ is \emph{coarsely Lipschitz} if for every $R\ge 0$ there exists $S\ge 0$ such that $d(x_1,x_2)\le R$ in $X$ implies $d(\varphi(x_1),\varphi(x_2))\le S$ in $Y$. Two coarsely Lipschitz mappings $\varphi,\psi:X\to Y$ define the same morphism in the coarse category (we say that $\varphi$ and $\psi$ are \emph{close}) if there exists $H\ge 0$ such that $d(\varphi(x),\psi(x))\le H$ for every $x\in X$. Then it is easy to check that coarsely Lipschitz maps determine a morphism of ind-objects $({\VR R X})_R\to ({\VR S Y})_S$ and if two mappings are close, then they induce homotopic maps between the ind-objects. Thus, the homotopy type of $({\VR k X})_k$ is a coarse invariant. This in particular implies that the ind-objects $(\tilde \hlgy_q(\VR k X))_k$ and $(\pi_1(\VR k X))_k$ are coarse invariants. Thus, it can be said that type $\Ftp m$ and type $\Flg m$ are coarse properties. In particular, type $\Ftp 1$ is equivalent to being coarsely connected and type $\Ftp 2$ is equivalent to being coarsely simply connected.

\section{Basic properties}
The geometric description obtained in Theorem~\ref{thm:metric} will 
ultimately help us find the criteria described in 
Sections~\ref{sec:crit1} and~\ref{sec:crit2}. It is also useful in 
finding short proofs for many well-known results, as described in this section.
 
In the course of this paper, we often translate the vertices of a simplex or chain with $t\in G$ of positive $\chi$-value. Since the metric is 
left-invariant, translating from the right or from the left is both useful but does fundamentally different things to the metric. We explain:
 \begin{enumerate}
  \item The translation $g\mapsto tg$ \emph{from the left} on the vertices of a simplex/chain does not change the distance, provided that the translation is simultaneous on each vertex:
  \[
   d(tg,th)=\ell((tg)^{-1}(th))=\ell(g^{-1}h)=d(g,h)
  \]
 does not change the distance. In this way, the Vietoris--Rips complex $\VR k G$ is 
a subset of $\E G$ that is $G$-invariant under left translation. If, on the 
other hand, $t_1$ and $t_2$ are two distinct non-central elements of $G$, then
 \[
  d(t_1g,t_2g)=\ell(g^{-1}t_1^{-1}t_2g)
 \]
 does not have an upper bound that depends solely on $t_1$ or $t_2$.
 \item We can translate \emph{from the right} $g\mapsto gt$ on the vertices of a simplex. This changes the distances $k$ between the vertices. Thus, the chain induced by this action lives in a different $\Delta^q_k$ for some $k$. But this change is by a controlled amount. That is,
 \begin{align*}
  d(gt,ht)
  &\le d(gt,g)+d(g,h)+d(h,ht)\\
  &=\ell((gt)^{-1}g)+d(g,h)+\ell((ht)^{-1}h)\\
  &=d(g,h)+2\ell(t)
 \end{align*}
If $t_1,t_2\in G$ are two distinct elements, then the distance
\[
  d(gt_1,ht_2)=\ell((gt_1)^{-1}ht_2)\le d(g,h)+\ell(t_1)+\ell(t_2).
 \]
  also changes by a controlled amount.
 \end{enumerate}

Two metric spaces have the \emph{same coarse type} if they are isomorphic in the category of coarsely Lipschitz maps modulo close.
 
 \begin{lem}
 \label{lem:GL}
  If $L\in \RR$, $G$ a countable group and $\chi\not=0$ then the coarse type of
  \[
   G_L:=\{g\in G\mid \chi(g)\ge L\}
  \]
 does not depend on $L$.
 \end{lem}
 \begin{proof}
  Suppose $K\ge L$. Choose $t\in G$ with $\chi(t)>0$. Then $n\chi(t)\ge K-L$ for some $n\in \NN$. Then 
  \begin{align*}
   \varphi:G_L&\to G_K\\
   g&\mapsto gt^n
  \end{align*}
 is well-defined, for if $\chi(g)\ge L$ then 
 \[
 \chi(gt^n)=\chi(g)+n\chi(t)\ge K-L+L=K.
 \]
 This mapping $\varphi$ is a coarse inverse of the inclusion $\iota:G_K\subseteq G_L$. Namely, $\varphi$ is coarsely Lipschitz since $d(g,h)\le k$ implies $d(gt^n,ht^n)\le k+2n\ell(t)$ and the map $\varphi\circ \iota:g\mapsto gt^n$ on $G_K$ is close to the identity on $G_K$ since $d(g,gt^n)\le n\ell(t)$. The same can be said about the map $\iota\circ\varphi:g\mapsto gt^n$ on $G_L$ and the identity on $G_L$.
 \end{proof}

 If $T\subseteq G$, then $S(G,T)$ denotes the set of characters on $G$ that vanish on $T$.
 \begin{lem}\cite[Lemma~5.2]{Bieri1988}
  If $G$ is a group of type $\Flg m$ with center $Z$ then $S(G,Z)^c\subseteq \Sigma^m(G;\ZZ)$.
 \end{lem}
 The superscript $c$ refers to the complement in the set of all characters
 \begin{proof}
  The case $m=1$ was covered in \cite[Proposition~A2.4]{Strebel2012}.
  
  We prove the case $m\ge 2$. Let $\chi\in S(G,Z)^c$ be a character and $k\ge 0$. Suppose that we already showed that $\chi\in \Sigma^{m-1}(G;\mathbb Z)$. Since $G$ is of type $\Flg m$ there exists some $l\ge 0$ such that 
$\homology{m-1} k G$ vanishes in $\homology {m-1} l G$. Since $\chi$ does not 
vanish on $Z$ there exists some $t\in Z$ with $\chi(t)>0$. We add $t$ to the set of generators if it is not already one, so that $\ell(t)=1$. Let $z\in \chains 
{m-1} k {G_\chi}$ be a cycle. Then it lives in $\chains {m-1} k G$ and so there exists some $c\in \chains m l G$ with $\diff m c=z$. 
  The map 
 \begin{align*}
  t:\ZZ[\Delta^q_l]=\chains q l G&\to \ZZ[\Delta^q_l]=\chains q l G\\
  d&\mapsto td
 \end{align*}
 is a chain endomorphism of $\ZZ G$-complexes (here we use the fact that $t$ is in the 
center) extending the identity on $\ZZ$. Then there is a chain homotopy $h$ of 
$\ZZ G$-complexes that join $t$ and $\id_{\ZZ[\Delta^q_l]}$ since they both extend 
the identity on $\ZZ$. Now we replace $c$ by 
\[
c':=c+\diff {m+1}h_mc=tc-h_{m-1}\diff mc=tc-h_{m-1}z.
\]
Then $\diff mc'=\diff m(c+\diff {m+1}h_mc)=z$. And 
\[
v(c')=v(tc-h_{m-1}z)\ge \min(\chi(t)+v(c),v(h_{m-1}z)).
\]
Then \cite[Lemma~2.1]{Bieri1988} implies that
 \[
 v(h_{m-1}(z))\ge v(h_{m-1}(z))-v(z)
 \]
has a uniform (on $z$) lower bound $K$. If $\chi(t)+v(c)$ is still 
below $K$ then we can repeat the procedure to obtain a chain $c$ with boundary $z$ and $v$-value at least $K$. Then $n\chi(t)\ge -K$ for some $n$. We obtain a chain $\tilde c$ from $c$ by replacing every vertex $g_i$ that is contained in a simplex occurring in $c$ and not in a simplex occurring in $\partial_mc$ by $g_it^n$. Then $\tilde c\in \chains m {l+2n} {G_\chi}$ and $\diff m\tilde c=z$. 
Thus, $\homology {m-1} k {G_\chi}$ vanishes in $\homology{m-1} {l+2n}{G_\chi}$. This shows that $(G_\chi,d)$ is of type $\mathrm{FP}_m$. Then Theorem~\ref{thm:metric} implies that $\chi\in \Sigma^m(G;\ZZ)$.
 \end{proof}
 
We recall that a metric space $X$ is of type $\Ftp 1$ if it is coarsely connected. It is of type $\Ftp m,m\ge 2$ if it is coarsely connected and $\pi_q(\VR r X)_r$ is essentially trivial for $q=1,\ldots,m-1$. 
 
 \begin{thm}
 \label{thm:homotopic_metric}
  If $\chi:G\to \RR$ is a character, then $\chi\in\Sigma^m(G)$ (the homotopical version) if and only if $G_\chi$ is of type $\Ftp m$. 
 \end{thm}
 \begin{proof}
   The geometric realization $Y=|\E G|$ of $\mathrm EG$ is a free contractible $G$-CW complex. Then $Y_r=((\Delta^q_r)_q)_r$ is a filtration of $Y$ by $G$-invariant cocompact subcomplexes. The map 
$v:(g_0,\dots,g_q)\mapsto \min_i(\chi(g_i))$ defines a valuation associated with $\chi$ on $Y_r$. Choose a generator $t\in G$ with $\chi(t)<0$. Then the double filtration $(t^k(\Delta_k^q\cap G_\chi^{q+1})_q)_k$ is essentially $(m-1)$-connected if and only if $\chi\in \Sigma^m(G)$ by an application of \cite[Theorem~4.7]{Hartmann2024c}. It remains to show that $\pi_{m-1}(t^k(\Delta_k^q\cap G_\chi^{q+1})_q)_k$ is essentially trivial if and 
only if $\pi_{m-1}((\Delta_k^q\cap G_\chi^{q+1})_q)_k$ is essentially trivial.
  
 Suppose $\pi_{m-1}(t^k(\Delta_k^q\cap G_\chi^{q+1})_q)_k$ is essentially trivial. If $k\ge 0$ then there exists some $l\ge 0$ with $\pi_{m-1}(t^k(\Delta^q_k\cap G_\chi^{q+1})_q)$ vanishes in $\pi_{m-1}(t^l(\Delta^q_l\cap G_\chi^{q+1})_q)$. Let $\varphi:S\to |\Delta_k^q\cap G_\chi^{q+1}|$ represent an element of $\pi_{m-1}((\Delta_k^q\cap G_\chi^{q+1})_q)$. Then $\varphi$ lives in $|t^k(\Delta_k^q\cap G_\chi^{q+1})_q|$. Thus, there exists a null-homotopy $H:S\times [0,1]\to |t^l(\Delta_l^q\cap G_\chi^{q+1})_q|$ of $\varphi$. Define a simplicial map $\alpha:t^l(\Delta_l^q\cap G_\chi^{q+1})_q\to (\Delta^q_{3l}\cap G_\chi^{q+1})_q$ by assigning $g\in G_\chi$ to $g$ and $g\not\in G_\chi$ to $gt^{-l}$. Then $|\alpha|\circ H$ is a homotopy joining $\varphi$ to a constant map. Thus, $\pi_{m-1}((\Delta_k^q\cap G_\chi^{q+1})_q)_k$ vanishes in $\pi_{m-1}((\Delta_{3l}^q\cap G_\chi^{q+1})_q)_k$. In this way, $\pi_{m-1}((\Delta_k^q\cap G_\chi^{q+1})_q)_k$ is essentially trivial.
  
Conversely, suppose that $\pi_{m-1}((\Delta_k^q\cap G_\chi^{q+1})_q)_k$ is essentially trivial. If $k\ge 0$ then there exists some $l\ge k$ with $\pi_{m-1}((\Delta_k^q\cap G_\chi^{q+1})_q)$ vanishes in $\pi_{m-1}((\Delta_l^q\cap G_\chi^{q+1})_q)$. Let $\varphi:S\to |t^k(\Delta_k^q\cap G_\chi^{q+1})_q|$ represent an element of $\pi_{m-1}(t^k(\Delta_k^q\cap G_\chi^{q+1})_q)$. Then $t^{-k}\varphi$ lives in $|\Delta_k^q\cap G_\chi^{q+1}|$. Thus, there exists $H:S\times [0,1]\to |\Delta^q_l\cap G_\chi^{q+1}|$ joining $\varphi$ to a constant map. Then $t^kH$ lives in $|t^l(\Delta_l^q\cap G_\chi^{q+1})_q|$ and homotopes $\varphi$ to a constant map. Thus, $\pi_{m-1}(t^k(\Delta_k^q\cap G_\chi^{q+1})_q)$ vanishes in $\pi_{m-1}(t^l(\Delta_l^q\cap G_\chi^{q+1})_q)$ In this way, we showed that $\pi_{m-1}(t^k(\Delta_k^q\cap G_\chi^{q+1})_q)_k$ is essentially trivial.
 \end{proof}
 
  \begin{lem}\cite[(3.5)~Satz~B]{Renz1988}
  If $m\ge 2$, then
   \[
    \Sigma^m(G)=\Sigma^2(G)\cap \Sigma^m(G,\ZZ).
   \]
  \end{lem}
 \begin{proof}
 We first prove $ \Sigma^m(G)\subseteq \Sigma^m(G,\ZZ)$. If $\chi\in \Sigma^m(G)$ then $\pi_q(\VR k{G_\chi})_k$ is essentially trivial for $q=0,\ldots,m-1$. We use \cite[Lemma~1.1.3]{Abels1997} which states that there exists a sequence $(B_k)_k$ of $(m-1)$-connected spaces such that $(\VR k{G_\chi})_k$ and $(B_k)_k$ are isomorphic as ind-spaces. Then $\tilde \hlgy_q(B_k)=0$ for every $k,q=0,\ldots,m-1$ according to the Hurewicz theorem \cite[Theorem~4.32]{Hatcher2002}. Thus, $\tilde \hlgy_q(\VR k{G_\chi})$ is essentially trivial. This proves $\chi\in \Sigma^m(G,\ZZ)$.
 
 Now, we prove $\Sigma^2(G)\cap \Sigma^m(G,\ZZ)\subseteq \Sigma^m(G)$. It is sufficient to prove that for every ind-space $(A_k)_k$ with both $\tilde\hlgy_q(A_k)_k$ essentially trivial for $q=0,\ldots,m-1$ and $\pi_0(A_k)_k$ and $\pi_1(A_k)_k$ essentially trivial we have that $\pi_q(A_k)_k$ is essentially trivial for $q=0,\ldots,m-1$. We do this by induction on $q$ 
starting with $q=1$. Since $\pi_0(A_k)_k,\pi_1(A_k)_k$ are essentially trivial, \cite[Lemma~1.1.3]{Abels1997} provides us with a sequence $(B^1_k)_k$ of $1$-connected spaces that is isomorphic to $(A_k)_k$ as an ind-space. Now we 
consider $m-1\ge q\ge 2$. We can assume that there exists a sequence 
$(B^{q-1}_k)_k$ of $(q-1)$-connected spaces which is isomorphic to $(A_k)_k$ as an ind-space. Then, according to the Hurewicz theorem, the Hurewicz homomorphism 
$h_q:\pi_q(B_k^{q-1})\to \hlgy_q(B_k^{q-1})$ is an isomorphism for every $k$. 

Since $\hlgy_q(B^{q-1}_k)_k$ is essentially trivial, there is for every $k$ some $l$ such that the left vertical map in 
\[
\xymatrix{
 \hlgy_q(B^{q-1}_k;\ZZ)\ar[r]^{h_q^{-1}}\ar[d]_0
 & \pi_q(B^{q-1}_k)\ar[d]\\
 \hlgy_q(B^{q-1}_l;\ZZ)\ar[r]^{h_q^{-1}}
 & \pi_q(B^{q-1}_l)
 }
\]
is zero. Since $h_q^{-1}$ is surjective, the right vertical map must be $0$. Thus, $(B^{q-1}_k)_k$ is essentially $q$-connected. By \cite[Lemma~1.1.3]{Abels1997} there is a sequence $(B^q_k)_k$ of $q$-connected spaces that is isomorphic to $(B^{q-1}_k)_k$ as ind-spaces. Since $(B^{q-1}_k)_k$ is isomorphic to $(A_k)_k$, this proves the induction hypothesis for the next step.
 \end{proof}

 \begin{lem}
  A finitely generated group $G$ is of type $\Flg 2$ if there is some $l\ge0$ such that $\homology {1} 1 G$ vanishes in $\homology {1} l G$.
 \end{lem}
\begin{proof}
 We first show that the hypothesis of the lemma is independent of the choice of generating set. Suppose $G$ is generated by a symmetric and finite set $X=X^{-1}$ and $Y=X\cup\{y,y^{-1}\}$ is obtained from $X$ by adding a redundant generator and its formal inverse.
 
  Suppose first that $\homology {1} 1 {G,d_X}$ vanishes in $\homology {1} l {G,d_X}$. We prove that there exists some $L$ such that $\homology{1} 1 {G,d_Y}$ vanishes in $\homology {1} L {G,d_Y}$. Suppose $y=y_1\cdots y_n$ as a word in $X$. Then $d(y,y_1\cdots y_k)\le n$ for every $k=1,\ldots,n$. Let $z\in \chains 1 1 {G,d_Y}$ be a cycle. Let $z'$ be obtained from $z$ by replacing each edge labeled $y$ by $y_1\cdots y_n$. So, if $(g_0,g_1)\in G^2$ is an edge with $g_0^{-1}g_1=y$, then we replace $(g_0,g_1)$ by $(g_0,g_0y_1)+(g_0y_1,g_0y_1y_2)+\cdots+(g_0y_1\cdots y_{n-1},g_1)$. And each edge labeled $y^{-1}$ is replaced by $y_n^{-1}\cdots y_1^{-1}$ in the same way. 
This is a cycle in $\chains 1 1 {G,d_X}$ and therefore there exists a chain $c$ in $\chains 2 l {G,d_X}$ with boundary $\diff 2 c=z'$. 

Without loss of generality, we replaced only one edge in the above construction; otherwise, we repeat the following procedure for the remaining edges. We define 
\[
c':=c-(g_0,g_0y_1,g_1)-(g_0y_1,g_0y_1y_2,g_1)-\cdots -(g_0y_1\cdots y_{n-2},g_0y_1\cdots y_{n-1},g_1).
\]
Then we compute
\begin{align*}
    \partial_2(c')
    =\,&\partial_2(c-(g_0,g_0y_1,g_1)-(g_0y_1,g_0y_1y_2,g_1)-\cdots -(g_0y_1\cdots y_{n-2},g_0y_1\cdots y_{n-1},g_1))\\
    =\,&z'-(g_0y_1,g_1)+(g_0,g_1)-(g_0,g_0y_1)\\
    &-(g_0y_1y_2,g_1)+(g_0y_1,g_1)-(g_0y_1,g_0y_1y_2)\\
    &-\vdots\\
    &-(g_0y_1\cdots y_{n-1},g_1)+(g_0y_1\cdots y_{n-2},g_1)-(g_0y_1\cdots y_{n-2},g_0y_1\cdots y_{n-1})\\
    =\,&z+(g_0,g_0y_1)+\cdots+(g_0y_1\cdots y_{n-1},g_1)-(g_0,g_1)\\
    &+(g_0,g_1)-(g_0,g_0y_1)-(g_0y_1,g_0y_1y_2)-\cdots-(g_0y_1\cdots y_{n-1},g_1)\\
    &-(g_0y_1\cdots y_{n-2},g_0y_1\cdots y_{n-1})\\
    =\,&z.
\end{align*}

In this way, we obtain a chain $c'\in\chains 2 {l+n} {G,d_Y}$ with boundary $\diff 2 c'=z$. Thus, $\homology 1 1 {G,d_Y}$ vanishes in $\homology 1 {l+n} {G,d_Y}$. 
 
 Now, suppose $G$ is assigned with the metric from $Y$ and $\homology 1 1 {G,d_Y}$ vanishes in $\homology 1 l {G,d_Y}$. Let $z\in \chains 1 1 {G,d_X}$ be a cycle. Then it also is a cycle in $\chains 1 1 {G,d_Y}$. Thus, there exists a chain $c\in \chains 2 l {G,d_Y}$ with $\diff 2 c =z$. Now $c\in \chains 2 {nl} {G,d_X}$ also. This shows that $\homology 1 1 {G,d_X}$ vanishes in $\homology 1 {nl} {G,d_X}$.
 
 Now, we show the implication of the lemma. Suppose $\homology 1 1 {G,d_X}$ vanishes in $\homology 1 l {G,d_X}$ and let $k\in \NN$ be a number. Every cycle in $\chains 1 k {(G,d_X)}$ is also a cycle in $\chains 1 1 {G,d_{X^{\le k}}}$ where $X^{\le k}$ describes the words in $X$ of length at most $k$. Since $X^{\le k}$ is obtained from $X$ by adding redundant generators, there exists some $L\ge 0$ such that $\homology 1 1 {G,d_{X^{\le k}}}$ vanishes in $\homology 1 L {G,d_{X^{\le k}}}$. Then $\homology 1 k {G,d_X}$ vanishes in $\homology 1 {kL} {G,d_X}$. This shows that $\homology 1 k G$ is essentially trivial. Thus, $G$ is of type $\Flg 2$.
\end{proof}

In fact, it can be shown that $\homology {1} l G=0$, that is, $\mathrm{VR}_l(G)^{(2)}$ is a $1$-acyclic, $G$-free simplicial set with finitely many $G$-orbits. This property is also known as type $\mathrm{FH}_2$ and is equivalent to type $\mathrm{FP}_2$ \cite[pg.~542]{Bridson2022}. Whether type $\mathrm{FP}_m$ implies type $\mathrm{FH}_m$ for $m>2$ is still open.

 \begin{lem}
  If $0\le K-L\in \im \chi$ and $0\le k\le l$, then $\homology q k {G_K}$ vanishes in $\homology q 
l {G_K}$ if and only if $\homology q k {G_L}$ vanishes in $\homology q l {G_L}$.
\end{lem}
\begin{proof}
  Suppose $t\in G$ with $\chi(t)=K-L$.
  
    First, suppose $\homology q k {G_K}$ vanishes in $\homology q l {G_K}$. Let $z\in \chains q k {G_L}$ be a cycle. Then $v(z)\ge L$ and thus
  \[
   v(tz)=\chi(t)+v(z)\ge K-L+L=K.
  \]
  Thus $tz\in \chains q k {G_K}$. Then there exists some $c\in \chains {q+1} l {G_K}$ with $\diff {q+1}c=tz$. Since $v(c)\ge K$,
  \[
   v(t^{-1}c)=-\chi(t)+v(c)\ge L-K+K=L.
  \]
  Thus, $t^{-1}c\in\chains {q+1} l {G_L}$ and $\diff {q+1}t^{-1}c=t^{-1}(tz)=z$. In this way, $\homology q k {G_L}$ vanishes in $\homology q l {G_L}$.
  
     Conversely, suppose that $\homology q k {G_L}$ vanishes in $\homology q l {G_L}$. Let $z\in \chains q k {G_K}$ be a cycle. Then $t^{-1}z\in \chains q k {G_L}$ by the same calculation as before. Then there exists some $c\in \chains {q+1} l {G_L}$ with $\diff {q+1}c=t^{-1}z$. Then $tc\in\chains {q+1} l {G_K}$ by the same calculation as before and $\diff {q+1}tc=t(t^{-1}z)=z$. In this way, $\homology q k {G_K}$ vanishes in $\homology q l {G_K}$.
 \end{proof}

If $a\in G$ we write $\Delta_a=\{(g,ga)|g\in G\}$. A mapping $\varphi:X\to Y$ is \emph{coarsely injective} if for every $S\ge 0$ there exists $R\ge 0$ such that for every $x,x'\in X$: $d(\varphi(x),\varphi(x'))\le S$ implies $d(x,x')\le R$. The mapping $\varphi$ is \emph{coarsely surjective} if there exists $R\ge 0$ such that $\Delta_R[\im \varphi]=Y$, where $\Delta_R[\im\varphi]:=\{y\in Y\mid \exists z\in \im\varphi:d(y,z)\le R\}$. A map between metric spaces is a \emph{coarse equivalence} if it is coarsely Lipschitz, coarsely injective, and coarsely surjective. If the map is between countable groups $G$ and $H$, then we take the word length metric on $G$ and $H$ according to a finite generating set of a finitely generated group that $G$ or $H$ respectively embed in. 

\begin{prop}
 If $0\to N\to G\to Q\to 1$ is an exact sequence of countable groups with $N$ abelian, then the sequence admits a coarsely Lipschitz transversal $t:Q\to G$ if and only if there exists a coarse equivalence $\varphi:N\rtimes Q\to G$ (which does not need to be a group homomorphism) that fits into the diagram
 \[
  \xymatrix{
  0\ar[r]
  &N\ar[r]\ar@{=}[d]
  &N\rtimes Q\ar[d]^\varphi\ar[r]
  &Q\ar[r]\ar@{=}[d]
  &1\\
  0\ar[r]
  &N\ar[r]
  &G\ar[r]
  &Q\ar[r]
  &1.
  }
 \]
\end{prop}
Note that the projection is not split, that is, $t$ is not required to be a group homomorphism. The action of $x\in Q$ on $N$ is given by conjugation with $t(x)$.
\begin{proof}
 Suppose that  there exists a coarse equivalence $\varphi$ that fits into the above diagram. Then there exists a split $t':Q\to N\rtimes Q$. Then $t:=\varphi\circ t'$ is a coarsely Lipschitz transversal.
 
 Now, suppose $t:Q\to G$ is a coarsely Lipschitz transversal. Recall that $G=N\times Q$ with the group operation given by
 \[
  (a_1,x_1)\cdot(a_2,x_2)=(a_1\varphi(x_1)(a_2)f(x_1,x_2),x_1x_2)
 \]
where $\varphi(x)(a):=t(x)at(x)^{-1}$ and $f(x_1,x_2)=t(x_1)t(x_2)t(x_1x_2)^{-1}$. Conversely, the semidirect product $N\rtimes Q$ can be written as $N\times Q$ with the group operation given by
\[
  (a_1,x_1)\ast(a_2,x_2)=(a_1\varphi(x_1)(a_2),x_1x_2).
 \]
 Now we check that the identity on $N\times Q$ viewed as a map $\alpha:G\to N\rtimes Q$ is both coarsely Lipschitz and coarsely injective. Let $(a,x)\in G$. If $t^{\times 2}(\Delta_{x^{-1}})\subseteq \Delta_{g_1}\cup\cdots\cup\Delta_{g_n}$, then we show $\alpha^{\times2}(\Delta_{(a,x)})\subseteq \Delta_{(at(x)^{-1}g_1,x)}\cup\cdots \cup\Delta_{(at(x)^{-1}g_n,x)}$. Let $((a_1,x_1),(a_2,x_2))\in \Delta_{\cdot(a,x)}$. Then
 \begin{align*}
  (a,x)
  &=(a_1,x_1)^{-1\cdot}\cdot(a_2,x_2)\\
  &=(t(x_1)^{-1}a_1^{-1}t(x_1^{-1})^{-1},x_1^{-1})\cdot(a_2,x_2)\\
  &=(t(x_1)^{-1}a_1^{-1}t(x_1^{-1})^{-1}t(x_1^{-1})a_2t(x_1^{-1})^{-1}t(x_1^{-1})t(x_2)t(x_1^{-1}x_2),x_1^{-1}x_2)\\
  &=(t(x_1)^{-1}a_1^{-1}a_2t(x_2)t(x_1^{-1}x_2),x_1^{-1}x_2).
 \end{align*}
Now
\begin{align*}
 (a_1,x_1)^{-1\ast}\ast(a_2,x_2)
 &=(t(x_1)^{-1}a_1^{-1}t(x_1),x_1^{-1})\ast(a_2,x_2)\\
 &=(t(x_1)^{-1}a_1^{-1}t(x_1)t(x_1^{-1})a_2t(x_1^{-1})^{-1},x_1^{-1}x_2)\\
 &=(t(x_1)^{-1}a_1^{-1}a_2t(x_1),x)\\
 &=(at(x)^{-1}t(x_2)^{-1}t(x_1),x)\\
 &\in\{(at(x)^{-1}g_1,x),\cdots,(at(x)^{-1}g_n,x)\}.
\end{align*}
Thus, $\alpha$ is coarsely Lipschitz. Now we show that $\alpha$ is coarsely injective. Let $(b,x)\in N\rtimes Q$. If $t^{\times 2}(\Delta_{x})\subseteq \Delta_{h_1}\cup\cdots\cup\Delta_{h_n}$, then we show 
\[
(\alpha^{-1})^{\times2}(\Delta_{(b,x)})\subseteq \Delta_{(ah_1t(x),x)}\cup\cdots \cup\Delta_{(ah_nt(x),x)}.
\]
Let $((a_1,x_1),(a_2,x_2))\in \Delta_{\ast(b,x)}$. Now $b=t(x_1)^{-1}a_1^{-1}a_2t(x_1)$ and $x=x_1^{-1}x_2$. Then
\begin{align*}
  (a_1,x_1)^{-1\cdot}\cdot(a_2,x_2)
  &=(t(x_1)^{-1}a_1^{-1}a_2t(x_2)t(x_1^{-1}x_2),x_1^{-1}x_2)\\
  &=(t(x_1)^{-1}a_1^{-1}a_2t(x_1)t(x_1)^{-1}t(x_2)t(x_1^{-1}x_2),x)\\
  &=(bt(x_1)^{-1}t(x_2)t(x),x)\\
  &\in \{(bh_1t(x),x),\ldots,(bh_nt(x),x)\}.
 \end{align*}
 Thus, $\alpha$ is coarsely injective.
\end{proof}

\begin{lem}
 If $\pi:G\to Q$ is a quotient with a coarsely Lipschitz transversal and $\chi$ is a character on $Q$ with $\chi\circ \pi\in\Sigma^m(G,\ZZ)$ [$\chi\circ\pi\in\Sigma^m(G)$], then $\chi\in \Sigma^m(Q,\ZZ)$ [$\chi\in \Sigma^m(Q)$, respectively].
\end{lem}
Compare with 
\cites[Corollary~2.8]{Meinert1997}[Corollary~3.14]{Meinert1996}[Theorem~8]{
Alonso1994}[Proposition~2.8]{Almeida2018}.
\begin{proof}
 Suppose that $\chi\circ\pi\in\Sigma^m(G,\ZZ)$ or that $\chi\circ\pi\in\Sigma^m(G)$, respectively. Since $g\in G_{\chi\circ\pi}$ if and only if $\chi\circ\pi(g)\ge 0$ if and only if $\pi(g)\in Q_\chi$, we obtain $G_{\chi\circ\pi}\subseteq \pi^{-1}(Q_\chi)$. Also, if $g,g'\in G$ with $\pi(g)=\pi(g')$, then $\chi\circ\pi(g)=\chi\circ\pi(g')$. So any two elements in the same fiber of $\pi$ have the same valuation and $\pi':=\pi|_{G_{\chi\circ\pi}}:G_{\chi\circ\pi}\to Q_\chi$ is well defined.
 
  Let $t:Q\to G$ be the coarsely Lipschitz transversal. So $\pi\circ t=\id_Q$. If $q\in Q_\chi$, then $\pi\circ t(q)=q\in Q_\chi$ which implies $t(q)\in \pi^{-1}(Q_\chi)=G_{\chi\circ \pi}$. So the composition of $\pi'$ with $t':=t|_{Q_\chi}$ is well-defined and $\pi'\circ t'=\id_{Q_\chi}$. By assumption $G_{\chi\circ\pi}$ is of type $\Flg m$ or of type $\Ftp m$, respectively (see Theorem~\ref{thm:metric}). Then the homology groups or homotopy groups of $(\VR k{G_{\chi\circ\pi}})_k$ as an ind-object vanish in dimensions $q=0,\ldots,m-1$. Now, $\id_{Q_\chi}$ factors over $G_{\chi\circ\pi}$ via $\pi'$ and $t'$, both of which are coarsely Lipschitz and therefore induce morphisms of ind-objects. This implies that the homology or homotopy groups of $(\VR k{Q_\chi})_k$ also vanish in dimensions $q=0,\ldots,m-1$. This means $Q_\chi$ is of type $\Flg m$ or of type $\Ftp m$, respectively, and Theorem~\ref{thm:metric} implies $\chi\in \Sigma^m(G,\ZZ)$ or $\chi\in\Sigma^m(G)$, respectively.
\end{proof}
\section{Criteria for homological finiteness properties}
\label{sec:crit1}
This section proves Theorem~\ref{thma:mulg}, Theorem~\ref{thma:varphilg} and Theorem~\ref{thma:algo1}. 
In addition, Procedure~\ref{algo:sigmam} is listed here.
\begin{lem}
\label{lem:chainendoreadable}
If $G$ is a countable group and $\varepsilon_*:{\chains * k G}^{(m)}\to {\chains * k G}^{(m)}$ is a chain 
endomorphism of chain complexes of $\ZZ G$-modules extending the identity on $\ZZ$, then there exist $l\ge k$ and a chain 
homotopy $\lambda_*:{\chains * k G}^{(m)}\to {\chains {*+1} l G}^{(m+1)}$ joining $\varepsilon_*$ to the inclusion ${\chains * k G}^{(m)}\to {\chains * l G}^{(m)}$.
\end{lem}
\begin{proof}
 For each $q\ge -1$ set $\sigma_q:=\varepsilon_q-\id_q$. Since $\sigma_{-1}=0$ 
we let $\lambda_{-2}:=0,\lambda_{-1}:=0$. Then
 \[
  \sigma_{-1}=0=\lambda_{-2}\circ \partial_{-1}+\partial_0\circ \lambda_{-1}.
 \]
Now suppose $m\ge q\ge 0$ and $\lambda_{-2},\ldots,\lambda_{q-1}$ have been constructed. We prove that $\mbox{im}(\sigma_q-\lambda_{q-1}\circ\diff q)\subseteq \ker \partial_q$:
\begin{align*}
 \partial_q\circ(\sigma_q-\lambda_{q-1}\circ \partial_q)
 &=\partial_q\circ\sigma_q-\partial_q\circ \lambda_{q-1}\circ \partial_q\\
 &=\partial_q\circ\sigma_q-(\sigma_{q-1}-\lambda_{q-2}\circ \partial_{q-1})\circ\partial_q\\
 &=\partial_q\circ\sigma_q-\sigma_{q-1}\circ\partial_q\\
 &=0.
\end{align*}
Thus, there is a diagram
\[
 \xymatrix{
 &\chains q k G\ar[d]^{\sigma_q-\lambda_{q-1}\circ\partial_q}\ar@{..>}[dl]_{\exists \lambda_q}\\
 \ZZ[G^{q+2}]\ar[r]_{\partial_{q+1}}
 &\diff q (\ZZ[G^{q+2}]).
 }
\]
Since $\chains q k G$ is projective and $\partial_{q+1}$ is surjective onto its image, there exists a $\ZZ G$-linear map $\lambda_q:\chains q k G\to \ZZ[G^{q+2}]$ such that $\partial_{q+1}\circ \lambda_q=\sigma_q-\lambda_{q-1}\circ\partial_q$. Hence, $\varphi_q-\id_q=\sigma_q=\partial_{q+1}\circ \lambda_q+\lambda_{q-1}\circ \partial_q$ as required. Since $\chains q k G$ is finitely generated free, the image of $\lambda_q$ lies in some $\chains {q+1} l G$. So we have defined $\lambda_q$ on $\chains q k G$.
\end{proof}

\begin{thm}
\label{thm:sigmam-crit}
 If $G$ is a group of type $\Flg m$ and $\chi:G\to \RR$ is a non-zero character, then the following are equivalent:
 \begin{enumerate}
 \item $\chi\in \Sigma^m(G;\ZZ)$;
 \item there exists $(k_0,\ldots,k_m)\in \NN^{m+1}$ so that $\tilde 
\hlgy_q(\VR {k_q}{G_\chi})$ vanishes in $\tilde \hlgy_q(\VR{k_{q+1}}{G_\chi})$ 
for each $q=0,\ldots,m-1$;
 \item for every $k\ge 0$ there exists a chain endomorphism of $\ZZ G$-complexes 
 \[
 \varphi_*:{\chains * k G}^{(m)}\to {\chains * k G}^{(m)}
\]
extending the identity on $\ZZ$ with $v(\varphi_q(c))-v(c)>0$ for every $c\in \chains q k G$ and $q=0,\ldots,m$.
 \end{enumerate}
\end{thm}
\begin{proof}
That condition 1 implies condition 2 is obvious.

We now show that condition 2 implies condition 3. We construct $\varphi_q$ inductively with the property $\im\varphi_q\subseteq \mathrm C_q(\VR {k_q}G)$ starting with $q=0$. Suppose $t\in G$ is a generator with $\chi(t)>0$. Then $\varphi_0$ is defined to send $g$ to $gt$. This map raises the value of $v$ by $\chi(t)$ and has image in $\ZZ[\Delta_0^0]$. Now we construct $\varphi_1$. If $\bar x=(1,g_1)\in \Delta_k^1$, then there exists a path $a_0=t,\cdots,a_n=g_1t$ in $G_{\min(\chi(t),\chi(g_1t))}$ joining $t$ to $g_1t$. Then 
\[
 \varphi_1(\bar x):=\sum_{i=0}^{n-1}(a_i,a_{i+1}).
\]
The map $\varphi_1$ increases the value of $v$ by $\chi(t)$ and has image in $\ZZ[\Delta_1^1]$. Now, suppose that the homomorphism $\varphi_{q-1}$ has been constructed such that the image of $\varphi_{q-1}$ lies in $\ZZ[\Delta^{q-1}_{k_{q-1}}]$ and the $v$-value is raised by $\chi(t)$. Let $\bar x=(1,g_1,\cdots,g_q)\in \Delta^q_k$. Then $\varphi_{q-1}\circ \partial_q (\bar x)\in \ZZ[\Delta^{q-1}_{k_{q-1}}]$ with $v$-value increased by $\chi(t)$. Since $\diff {q-1}\circ \varphi_{q-1} \circ\diff q=\varphi_{q-2}\circ\diff {q-1}\circ\diff q=0$ we deduce that $\varphi_{q-1}\circ\partial_q(\bar x)\in \ker \diff {q-1}$.  Then there exists some $\bar y\in \chains q {k_q} {G_{v(\bar x)+\chi(t)}}$ with $\diff q \bar y=\varphi_{q-1}\circ\diff q(\bar x)$. Define $\varphi_q(\bar x):=\bar y$. The map $\varphi_q$ increases the value of $v$ by $\chi(t)$ and has image in $\ZZ[\Delta_{k_q}^q]$.

Now, we show that condition 3 implies condition 1. Suppose that there exists a chain endomorphism $\varphi$ that increases the $\chi$-value. Let $k\ge0$ be an index. By Lemma~\ref{lem:chainendoreadable} there exists a chain homotopy $\lambda:\chains{m-1} k G\to \chains m {l_1} G$ joining $\varphi$ to the identity. By \cite[Lemma 2.1]{Bieri1988} 
\[
 v(\lambda(c))-v(c)\ge 
\min_{x=(1_G,g_1,\ldots,g_{m-1})\in\Delta^{m-1}_k}(v(\lambda(x))-v(x))=:K
\]
for every $c\in \chains {m-1} k G$.

Let $z\in \chains {m-1} k {G_\chi}$ be a cycle. Then, in particular, $v(\lambda(z))\ge v(\lambda(z))-v(z)\ge K$. Since $G$ is of type $\Flg m$ there exists $c\in \chains m {l_2} G$ with $\diff m c=z$. Define
\[
 R:=\min_{x=(1_G,g_1,\ldots,g_m)\in\Delta^m_{l_2}}(v(\varphi_m(x))-v(x))>0.
\] 
 Then $(n+1)R\ge -v(c)$ for some $n\in\mathbb N$.

We define
\[
 c_0:=\lambda(z),\quad c_i:=\varphi_m^{\circ i}(c_0),\quad 
z_i:=\varphi_{m-1}^{\circ i}(z)
\]
where $\circ i$ denotes the $i$-fold composition of functions. Then
\[
 \diff m {c_i}=\varphi^{\circ i}_{m-1}\diff m {c_0}= \varphi^{\circ i}_{m-1}\circ\diff m \circ \lambda(z)=\varphi_{m-1}^{\circ i}(z-\varphi_{m-1}(z))=z_i-z_{i+1}.
\]
Define
\[
 \tilde c:=\sum_{i=0}^nc_i+\varphi^{\circ n+1}_m(c).
\]
Then
\[
 \diff m \tilde c=\sum_{i=0}^n\diff m c_i+\varphi^{\circ n+1}_{m-1}(\diff m c)=\sum_{i=0}^n(z_i-z_{i+1})+z_{n+1}=z
\]
and
\begin{align*}
 v(\tilde c)
 &\ge \min_i(v(c_i),\varphi^{\circ n+1}_m(c))\\
 &\ge \min (K,(n+1)R+v(c))\\
 &\ge \min (K,0).
\end{align*}
So $\homology {m-1} k {G_\chi}$ vanishes in $\homology {m-1} {\max(l_1,l_2)} {G_K}$. Thus, $\chi\in \Sigma^m(G,\ZZ)$.
\end{proof}

A vector $(n_0,\ldots,n_m)\in\mathbb N_0^{m+1}$ is a \emph{homological connecting vector for $G$} if $\tilde \hlgy_q(\VR{n_q}G)$ vanishes in $\tilde \hlgy_q(\VR{n_{q+1}}G)$ 
for $q=0,\ldots,m-1$.

Note that a countable group $G$ (that is not necessarily finitely generated) is endowed with the subspace metric of an inclusion $G\to G_1$ into a finitely generated group which, in turn, is endowed with the word length metric of a finite generating set.
\begin{thm}
\label{thm:chainendo-fpm}
 If $G$ is a countable group, then the following are equivalent:
 \begin{enumerate}
  \item $G$ is of type $\Flg m$;
  \item there exists a homological connecting vector $(n_0,\ldots, n_m)$ for $G$;
  \item there exists $K_0\ge 1$ and for every $k\ge K_0$ a chain homomorphism of $\ZZ G$ complexes 
  \[
  \mu_*:\mathrm C_*(\mathrm{VR}_k(G))^{(m)}\to \mathrm C_*(\mathrm{VR}_{K_0}(G))^{(m)}
  \]
  extending the identity on $\ZZ$.
 \end{enumerate}
\end{thm}
 \begin{proof}
 That condition 1 implies condition 2 is obvious. 
 
  We now assume condition 2 and show condition 3. That is, we construct $\mu$ inductively on the $q$-skeletons with the property that $\im \mu_q\subseteq \mathrm C_q(\VR {n_q} G)$ starting with $q=0$. If $q=0$ then $\mu_0$ is defined as the identity on $\ZZ[G]$. Observe that $\im\mu_0\subseteq \chains 0 0 G$. 
  
  If $\mu_0,\ldots,\mu_{q-1}$ have been constructed and $(1_G,g_1,\ldots,g_q)\in \Delta^q_k$, then 
  \[
  \mu_{q-1}\circ \partial_q(1_G,g_1,\ldots,g_q)\in \chains {q-1} {n_{q-1}} G.
  \]
  Since $\homology {q-1} {n_{q-1}} G$ vanishes in $\homology{q-1}{n_q}{G}$ there exists $\bar y\in \chains q {n_q} G$ with $\diff q \bar y=\mu_{q-1}\circ \diff q(1_G,g_1,\ldots,g_q)$. Define $\mu'_q(1_G,g_1,\ldots,g_q):=\bar y$. Then $\mu_q$ is obtained from $\mu'_q$ by extending $G$-equivariantly. Then $\im\mu_q\subseteq \chains q {n_q} G$, as required. If $q=m$, then we define $K_0:=n_m$.

  Now, we show that condition 3 implies condition 1. Suppose that we already showed that $G$ is of type $\Flg {m-1}$. Let $l\ge 0$ be a number. Since $\mu_q$ extends the identity on $\ZZ$ there exist by Lemma~\ref{lem:chainendoreadable} a number $n\ge l$ and a chain homotopy $\lambda_q:\chains q l G \to \chains {q+1} n G$ that joins the inclusion to $\mu$. If $z\in \chains {m-1} l G$ is a cycle then since $(\ZZ[G^{q+1}],\partial_q)$ is acyclic there exists $c\in \ZZ[G^{m+1}]$ with $\diff m c=z$. Then
  \[
   \diff m(\mu_m(c)+\lambda_{m-1}(z))=\mu_{m-1}\diff m c+\diff m \circ 
\lambda_{m-1}(z)=\mu_{m-1}(z)+z-\mu_{m-1}(z)=z.
  \]
 This shows that $z$ is the boundary of $\mu_m(c)+\lambda_{m-1}(z)\in \mathrm C_m(\VR n G)$. In this way, $\homology {m-1} l G$ vanishes in $\homology {m-1} n G$, which implies that $G$ is of type $\Flg m$.
 \end{proof}
\SetAlgorithmName{Procedure}
  \IncMargin{1em}
 \begin{algorithm}[t]
  \SetKwInOut{Input}{input}
  \SetKwInOut{Output}{output}
  \Input{A group $G$ of type $\mathrm{FP}_m$ with a homological connecting vector $(n_0,\ldots,n_m)$, a non-zero character $\chi:G\to \RR$}
  \Output{``yes, $\chi\in \Sigma^m(G;\ZZ)$''}
  Define $n:=n_m$\;
  Pick $t\in G$ with $\chi(t)>0$\;
  Define a $\ZZ G$-homomorphism $\varphi_0:\ZZ[G]\to \ZZ[G]$ by $1_G\mapsto t$\;
  \For{$q\leftarrow 1$ \KwTo $m$}{
  \ForEach{$\bar x:=(1_G,g_1,\ldots,g_q)\in \Delta^q_n$}{
 search for $\bar y\in \ZZ[\Delta^q_n\cap G_{\chi(t)+v(\bar x)}^{q+1}]$ with 
$\diff q \bar y=\varphi_{q-1}\circ \diff q \bar x$\;
  define $\varphi_q(\bar x):=\bar y$\;
  }
  }
 \KwRet{``yes''}\;
 \vspace{0.1cm}
  \caption{$\Sigma^m$-criterion, homological version}
  \label{algo:sigmam}
  \end{algorithm}
Both for-loops in Procedure~\ref{algo:sigmam} run over a finite set: 
$\{1,\ldots,m\}$ and $\{(1_G,g_1,\ldots,g_q)\in \Delta^q_n\}$ are computable 
and finite.

At some point in the procedure one needs to compute an $N$-ball $\Delta_N[1_G]$ around $1_G$ in the Cayley-graph $\Gamma(G,X^{\le n})$ where $X^{\le n}$ are 
words of length at most $n$ and $N\ge 0$ is some large number. This problem is 
equivalent to computing the word length of a group element 
\cite[Proposition~5.7]{Meier2008}. So, if the word problem is solvable, then we can construct the Cayley-graph. Otherwise, we can guess parts of the Cayley graph, which is good enough in the scope of this discussion.
  
Note that inside the nested for-loops we let the function \emph{search} run for a possibly infinite amount of time. If $0\not=\chi\in \Sigma^m(G;\ZZ)$, then the search for $\bar y$ stops for every homological connecting vector $(k_0,\ldots,k_m)$ of $G$. We can make that precise:
  
\begin{thm}
\label{thm:algo1}
 Let $G$ be a group with homological connecting vector $(k_0,\ldots,k_m)$ and solvable word-problem. If $\chi:G\to \RR$ is a non-zero character and $k:=k_m$, then the following are equivalent:
 \begin{enumerate}
  \item $\chi\in\Sigma^m(G;\ZZ)$;
  \item there exists a chain endomorphism of $\ZZ G$-complexes 
  \[
  \varphi_*:{\chains * k G}^{(m)}\to {\chains * k G}^{(m)}
  \]
  extending the identity on $\ZZ$ with $v(\varphi_q(c))-v(c)>0$ for every $c\in \chains q k G$ and $q=0,\ldots,m$;
  \item Procedure~$\ref{algo:sigmam}$ terminates with the output ``yes''.
 \end{enumerate}
\end{thm}
\begin{proof}
We first show that condition 1 implies condition 2. Suppose condition 1 holds, so $\chi\in \Sigma^m(G;\ZZ)$. Then Theorem~\ref{thm:sigmam-crit} implies that there exists $l\ge 0$ and for every $n\ge l$ a chain endomorphism $\varphi_{n,*}:\chains * n G\to \chains * n G$ that extends the identity on $\ZZ$, which increases the valuation and $\im \varphi_q \subseteq \chains q l G$ for every $q=0,\ldots,m$. If $k\ge l$, then denote by $\iota_{lk}$ the inclusion $\chains * l G \subseteq \chains * k G$. Then $\iota_{lk}\circ \varphi_{k,*}$ is the desired chain endomorphism. If conversely $k<l$ then since $(k_0,\ldots,k_m)$ is a homological connecting vector for $G$ Theorem~\ref{thm:chainendo-fpm} implies for every $n\ge k$ that there exists a chain endomorphism $\mu_{n,*}:\chains * n G\to \chains * n G$ extending the identity on $\ZZ$ with $\im \mu_q\subseteq \chains q k G$ for every $q=0,\ldots,m$. Then
\begin{align*}
 \inf_{\substack{c\in\chains q l G\\q=0,\ldots,m}} (v(\mu_q(c))-v(c))
 &\ge \min_{\substack{\sigma\in \Delta^q_l\cap (1_G\times G^q)\\q=0,\ldots,m}}(v(\mu_q(\sigma))-v(\sigma))\\
 &=:K\in \mathbb R
\end{align*}
Then $i\chi(t)\ge -K$ for some $i\in \mathbb N$. Denote by $\iota_{kl}$ the inclusion $\chains * k G \subseteq \chains * l G$. Then $\mu_{l,*}\circ \varphi_{l,*}^{\circ i+1}\circ \iota_{kl}$ is the required chain endomorphism:
\begin{align*}
 v(\mu_q\circ \varphi_{l,q}^{\circ i+1}\circ \iota_{kl}(c))
 &\ge K+v(\varphi_{l,q}^{\circ i+1}\circ \iota_{kl}(c))\\
 &\ge K+(i+1)\chi(t)+v(\iota_{kl}(c))\\
 &\ge \chi(t)+v(c)\\
 &>v(c).
\end{align*}

We now assume condition 2 and show condition 1. Let $n\ge k$. Since $(k_0,\ldots,k_m)$ is a homological connecting vector for $G$, there exists a chain endomorphism $\mu_*:\chains * n G\to \chains * n G$ extending the identity on $\ZZ$ with $\mbox{im }\mu_q\subseteq \chains q k G$ for each $q=0,\ldots,m$. Condition 2 also provides us with a chain endomorphism $\varphi_*:\chains q k G\to \chains q k G$ extending the identity on $\ZZ$ with $v(\varphi(c))-v(c)\ge \chi(t)$ for every $c\in \chains q k G$ and $q=0,\ldots,m$. As before,
 \[
  v(\mu(x))-v(x)\ge K
 \]
for every $x=(1,g_1,\ldots,g_q)\in\Delta^q_n,q=0,\ldots,m$ for some $K\in \RR$. Then $i\chi(t)\ge -K$ for some $i\in \NN$. Then $\varphi_*^{\circ i+1}\circ \mu_*:\chains * n G\to \chains * n G$ is a chain endomorphism extending the identity on $\ZZ$ with
\[
 v(\varphi_q^{\circ i+1}\circ \mu_q(c))\ge (i+1)\chi(t)+v(\mu_q(c))\ge (i+1)\chi(t)+K+v(c)\ge \chi(t)+v(c)
\]
and $\mbox{im }\varphi_q^{\circ i+1}\circ \mu_q \subseteq \chains q k G$ for $q=0,\ldots,m$. By Theorem~\ref{thm:sigmam-crit}, this proves $\chi\in \Sigma^m(G;\ZZ)$.
 
 If we assume condition 3, then Procedure~\ref{algo:sigmam} terminates with output ``yes''. The data collected during the computation provide us with a chain endomorphism that satisfies the requirements of condition 2, for we have specified it on a basis of the free $\mathbb Z[G]$-module $\mathbb Z[\Delta_k^q]$.

Now, if we assume condition 2, then the chain endomorphism sought in Procedure~\ref{algo:sigmam} exists. Since the word problem is solvable in $G$, we can construct the Cayley graph and ultimately also the Vietoris-Rips complex, which are countable data since the group is finitely generated. In this way, the function \emph{search} terminates, as does Procedure~\ref{algo:sigmam}.
\end{proof}

If a chain homomorphism $\varphi$ satisfies condition~2 in Theorem~\ref{thm:algo1} for a non-zero character $\chi:G\to \mathbb R$, then $\varphi$ is a \emph{witness for $\chi\in \Sigma^m(G,\mathbb Z)$}. Now, we introduce coordinates in order to determine for which set of non-zero characters a given chain endomorphism $\varphi$ is a witness.

If $X:=\{x_0,\ldots,x_n\}\subseteq G$ is a generating set for $G$, then we define a mapping 
\begin{align*}
 u:\Hom(G,\RR)&\to \RR^n\\
 \chi&\mapsto \begin{pmatrix}\chi(x_1)\\\vdots\\\chi(x_n)\end{pmatrix}
\end{align*}
which is continuous by definition. We also define a mapping
\begin{align*}
 w_X:\F(X)&\to \ZZ^n\\
 w&\mapsto \begin{pmatrix}\varepsilon_1\\\vdots\\\varepsilon_n \end{pmatrix}
\end{align*}
where $\varepsilon_i$ are the sum of the exponents of $x_i$ in $w$. And if $t:G\to \F(X)$ is a transversal of the usual projection, then we define $w:=w_X\circ t$. Then 
\[
 \langle w(g),u(\chi)\rangle=\chi(g)
\]
does not depend on the representation $t$. A vector $y\in \RR^n$ is in the image of $u$ if for every word $r$ in the defining relations 
\[
 \langle w_X(r),y\rangle=0.
\]
So $u(\Hom(G,\RR))$ is a closed linear subspace of $\RR^n$. We show that $\Sigma^m(G,\ZZ)$ is an open cone in $\Hom(G,\RR)$. If $c\in \ZZ[G^{q+1}]$, then denote by $c^{(0)}$ the vertices in $G$ appearing in the chain. If $\varphi_q:\chains q k G\to \chains q k G$ is a chain endomorphism extending the identity on $\ZZ$, then define a mapping 
\begin{align*}
 u_\varphi:\Hom(G,\RR)&\to \RR\\
 \chi&\mapsto \min_{\substack{\bar x=(1_G,g_1,\ldots,g_q)\in 
\Delta^q_k,\\q=0,\ldots,m}}\left(\min_{g\in \varphi_q(\bar x)^{(0)}}\langle 
w(g),u(\chi)\rangle-\min_{g\in \bar x^{(0)}}\langle w(g),u(\chi)\rangle\right).
\end{align*}
Now $u_\varphi$ is a continuous mapping, so 
\[
 U(\varphi):=u_\varphi^{-1}(0,\infty)
\]
is open.

\begin{prop}
 Let $G$ be a group with a homological connecting vector $(k_0,\ldots,k_m)$. If 
for 
$k:=\max(k_0,\ldots,k_m)$ there exists a chain endomorphism $\varphi_q:\chains 
q k G\to \chains q k G$ extending the identity on $\ZZ$, then 
$U(\varphi)\subseteq \Sigma^m(G,\ZZ)$.
\end{prop}
\begin{proof}
 Suppose $\chi\in U(\varphi)$. Then
 \begin{align*}
 0
 &<\min_{\bar x=(1_G,g_1,\ldots,g_q)\in \Delta^q_k,q=0,\ldots,m}\left(\min_{g\in \varphi_q(\bar x)^{(0)}}\langle w(g),u(\chi)\rangle-\min_{g\in \bar x^{(0)}}\langle w(g),u(\chi)\rangle\right)\\
 &=\min_{\bar x=(1_G,g_1,\ldots,g_q)\in \Delta^q_k,q=0,\ldots,m}\left(\min_{g\in \varphi_q(\bar x)^{(0)}}\chi(g)-\min_{g\in \bar x^{(0)}}\chi(g)\right)\\
 &=\min_{\bar x=(1_G,g_1,\ldots,g_q)\in 
\Delta^q_k,q=0,\ldots,m}\left(v(\varphi_q(\bar x))-v(\bar x)\right).
 \end{align*}
 Using a similar argument as in the proof of Theorem~\ref{thm:algo1} we obtain $\chi\in \Sigma^m(G,\ZZ)$.
\end{proof}

\begin{thm}\cite[Theorem~A]{Bieri1988}
 If $G$ is a group of type $\mathrm{FP}_m$, then the subset $\Sigma^m(G,\ZZ)$ is the cone with vertex at $0$ over an open subset in $\Hom(G,\RR)$.
\end{thm}
\begin{proof}
 If $\chi\not=0$ is a character on $G$ that belongs to $\Sigma^m(G,\ZZ)$, then there exists some chain endomorphism $\varphi$ that is a witness for $\chi$ belonging to $\Sigma^m(G,\ZZ)$. Then $U(\varphi)$ is an open neighborhood of $\chi$ in $\Sigma^m(G,\ZZ)$. So $\Sigma^m(G,\ZZ)\setminus\{0\}$ is open. Since for every $\lambda>0$ we have $\lambda\chi\in\Sigma^m(G,\ZZ)$ if and only if $\chi\in\Sigma^m(G,\ZZ)$ and also $0\in \Sigma^m(G,\ZZ)$ if $\Sigma^m(G,\ZZ)\not=\emptyset$ the set $\Sigma^m(G,\ZZ)$ is a cone with vertex at $0$.
\end{proof}
\section{Criteria for homotopical finiteness properties}
\label{sec:crit2}
This section proves Theorem~\ref{thma:mutp}, Theorem~\ref{thma:varphitp} and Theorem~\ref{thma:algo2}. In addition, Procedure~\ref{algo:sigma2} is listed here.

 In what follows, we frequently use the simplicial approximation theorem \cite[Theorem~2C.1]{Hatcher2002}. So it makes sense to work with $\Map(A,B)$ as a morphism between simplicial sets $A,B$, where $\Map(A,B)$ denotes simplicial maps $\varepsilon:A'\to B$ with $A'$ obtained from $A$ by finitely many barycentric subdivisions. In fact, the maximum length of a minimal edge path in the domain of $\varepsilon|_\sigma$ is uniformly bounded in $\sigma\in A$ and is the \emph{dept} of $\varepsilon$. The standard $q$-simplex is denoted by $\Delta^q$. We hope that this notation is not confused with $\Delta^q_k$ which denotes the set of $q$-simplices in $\VR k \cdot$. And $\pi_q(\cdot)$ on simplicial sets denotes the $q$th homotopy group of the geometric realization of the respective simplicial set.

\begin{lem}
\label{lem:homotopicalhomotopy}
If $G$ is a finitely generated group and $\varepsilon\in\Map({\VR k G}^{(m)},{\VR k G}^{(m)})$ a $G$-equivariant map, then there exist $l\ge k$ and a $G$-equivariant homotopy 
\[
\eta\in \Map({\VR k G}^{(m)}\times\Delta^1,{\VR l G}^{(m+1)})
\]
joining $\varepsilon$ to the inclusion $\VR k G\to \VR l G$.
\end{lem}
\begin{proof}
We construct $(n_0,\ldots,n_m)\in \mathbb N^{m+1}$ and $(l_0,\ldots,l_m)\in \mathbb R^{m+1}$ and 
\[
\eta\in \Map({\VR k G}^{(m)}\times\Delta^1,{\VR {l_m} X}^{(m+1)})
\]
inductively on the $q$-skeleton so that the domain of each $\eta_q$ has dept $n_q$ and $\im\eta_q\subseteq {\VR {l_q}G}^{(q+1)}$.

 We start with $q=0$. Since $\varepsilon_0$ is $G$-equivariant, $t:=\varepsilon_0(1_G)$ determines the map. Then $\eta_0(g,0):=g$ and $\eta_0(g,1):=gt$ is a $G$-equivariant simplicial map ${\VR k G}^{(0)}\times\Delta^1\to {\VR {l(t)} G }^{(1)}$. The two adjacent vertices in $\eta_0|_{g\times[0,1]}$ have distance at most $\ell(t)$. So we can define $l_0:=\ell(t)$. The dept of $\eta_0$ is $n_0:=1$.
 
 If $q>0$, then suppose that $\eta_0,\ldots,\eta_{q-1}$ have already been constructed. Let $\sigma\in(1_G\times G^q)\cap \Delta^q_k$ be a simplex representing a $G$-orbit in $\VR k G$. 

 Denote by $T_1$ the subdivision of the domain of $\varepsilon_q|_\sigma$, by $T_2$ the subdivision of the domain of $\id_q|_{\sigma}$ and by $T_3$ the subdivision of the domain of $\eta_{q-1}|_{\partial \sigma\times\Delta^1}$. 

Then $S$ is defined to be the union of $T_1$, $T_2$, and $T_3$ which intersect in $T_1|_{\partial\sigma}\cap T_3|_{\partial\sigma\times 1}$ and $T_2|_{\partial \sigma}\cap T_3|_{\partial\sigma\times 0}$. The boundary of this $S$ is empty, so the cone $T$ over $S$ with apex a single vertex has boundary the base $S$. Then the dept of $T$ is at most
    \[
    n_{q-1}+2+k+1=:n_q.
    \]

Then $\eta'_q|_\sigma$ is defined on $T$ in the following way. On $T_1$ it is $\varepsilon_q|_\sigma$, on $T_2$ it is $\id_q|_\sigma$, on $T_3$ it is $\eta_{q-1}|_{\partial \sigma\times[0,1]}$ and at the apex of the cone it is $1_G$. 

 The mapping $\eta'_q$ extends $\eta_{q-1}|_{\partial\sigma\times \Delta^1}$ and joins $\varepsilon_q$ to $\id_q$. Since $((1_G\times G^q)\cap \Delta^q_k)\times \Delta^1$ is a finite simplicial set, its image under $\eta'_q$ is also finite. Thus, $\im \eta'_q\subseteq {\VR {l_q} G }^{(q+1)}$ for some $l_q\ge 0$. 
 
 Then $\eta_q\in \Map({\VR k G}^{(q)}\times \Delta^1,{\VR {l_q} G}^{(q+1)})$ is obtained from $\eta'_q$ by extending $G$-equivariantly. If $\tau:=(y_0,\ldots,y_q)\in \Delta^q_k$ is a simplex, then $\sigma:=y_0^{-1}\tau\in (1_G\times G^q)\cap \Delta^q_k$. We check each restriction:
 \[
  \eta_q|_{\tau\times0}=\eta_q|_{y_0\sigma\times0}=y_0\eta_q|_{\sigma\times0}=\id_{y_0\sigma}=\id_\tau
 \]
and
\[
\eta_q|_{\tau\times1}=\eta_q|_{y_0\sigma\times1}=y_0\eta_q|_{\sigma\times1}=y_0\varepsilon_q|_\sigma=\varepsilon_q|_\tau
\]
and
\[
 \eta_q|_{\partial\tau\times\Delta^1}=y_0 \eta_q|_{\partial\sigma\times\Delta^1}=y_0\eta_{q-1}|_{\partial \sigma \times \Delta^1}=\eta_{q-1}|_{\partial\tau\times\Delta^1}. \qedhere
\]

\end{proof}

A vector $(k_0,\ldots, k_m)$ is a \emph{homotopical connecting vector for $G$} if $\pi_q(\VR{k_q}G)$ vanishes in $\pi_q(\VR{k_{q+1}}G)$ for every 
$q=0,\ldots,m-1$.

\begin{thm}
\label{thm:muhomotopy}
If $G$ is a countable group, then the following are equivalent:
\begin{enumerate}
 \item $G$ is of type $\Ftp m$;
 \item there exists a homotopical connecting vector $(k_0,\ldots,k_m)$ for $G$;
 \item there exists $K_0\ge 1$ and for every $k\ge K_0$ a $G$-equivariant mapping
 \[
 \mu\in \Map({\VR k G}^{(m)},{\VR{K_0}G}^{(m)});
 \]
\end{enumerate}
\end{thm}
\begin{proof}
 That condition 1 implies condition 2 is obvious.

  Now, we show that condition 2 implies condition 3. Suppose that condition 2 holds. There 
exists a homotopical connecting vector $(k_0,\ldots,k_m)$ for $G$. Let $k\ge 0$. We construct 
$\mu_q$ inductively on the skeletons with the property that $\im \mu_q\subseteq {\VR{k_q}G}^{(q)}$. On ${\VR k G}^{(0)}$ the mapping $\mu_0$ is defined to be the identity. 

If $\mu_0,\ldots,\mu_{q-1}$ have been constructed and 
$\sigma:=(1_G,x_1,\ldots,x_q)\in \Delta^q_k$ is a simplex, then $\im \mu_{q-1}|_{\partial\sigma}\subseteq \VR {k_{q-1}}G$. Since $\pi_{q-1}(\VR{k_{q-1}}G)$ vanishes in $\pi_{q-1}(\VR{k_q}G)$ there exists 
$\phi_\sigma\in \Map(\Delta^q,\VR{k_q}G)$ with $\phi_\sigma|_{\partial\Delta^1}=\mu_{q-1}|_{\partial\sigma}$. If $\tau:=(x_0,\ldots,x_q)\in 
\Delta^q_k$ is a simplex, then $\sigma:=(1_G,x_0^{-1}x_1,\ldots,x_0^{-1}x_q)$ is 
also a simplex with $x_0\sigma=\tau$. We define 
  \[
   \mu_q|_\tau:=x_0\phi_\sigma. 
  \]
This mapping has image in ${\VR{k_q}G}^{(q)}$ and extends $\mu_{q-1}$. If $q=m$, then we set $K_0:=k_m$.

We now show that condition 3 implies condition 1. Let $k\ge K_0$. Then there is a mapping $\mu\in \Map({\VR k G}^{(m)},{\VR{K_0}G}^{(m)})$. Suppose that we already showed that $G$ is of type $\Ftp {m-1}$. If $\alpha:S^{m-1}\to |\VR k G|$ is a continuous map, then there exists $\tilde\alpha\in \Map(\partial\Delta^m,\VR k G)$ homotopic to $\alpha$ by the simplicial approximation theorem. Then $\mu_{m-1}\circ \tilde\alpha$ has image in $\VR{K_0}G$. According to Lemma~\ref{lem:homotopicalhomotopy} there is $l\ge k$ so that $\mu\circ\tilde\alpha$ is homotopic to $\tilde\alpha$ with homotopy living in $\VR l G$. Since $\E G$ is contractible, $\tilde \alpha$ is homotopic to a constant mapping $\ast$ in $\E G$, denote by $\lambda$ the homotopy connecting them. Then $\mu\circ \tilde \alpha$ is homotopic to a constant mapping $\mu\circ \ast$. The homotopy connecting them, $\mu\circ \lambda$ has image in $\VR{K_0}G$. As a result, $\tilde \alpha$ is null-homotopic in $\VR l G$. Thus, $\pi_{m-1}(\VR k G)$ vanishes in $\pi_{m-1}(\VR l G)$. This implies that $G$ is of type $\Ftp m$.
\end{proof}

If $G=\langle x_1,\ldots,x_n\mid r_1,\ldots,r_m\rangle$ is given by a finite representation and $k\ge 1$, then $\mu$ can be read from the presentation as follows. Define $\mu_0$ to be the identity on $G$. Now, $\mu_1$ specifies for each $g_1\in G$ with $\ell(g_1)\le k$ a word $w_1$ in the generators $x_1,\ldots,x_n$ of length $\le k$. 
Now we construct $\mu_2$. The boundary of a simplex $(1_G,g_1,g_2)\in\Delta_k^2$ is sent by $\mu_1$ to the word $w_1$ (representing $g_1$) composed with the word $w_{12}$ (representing $g_1^{-1}g_2$) composed with the word $w_2^{-1}$ (representing $g_2^{-1}$). Then the word $w_1w_{12}w_2^{-1}$ is a relation in $G$, so it can be written as 
\[
w_1w_{12}w_2^{-1}=\prod_iy_i^{-1}r_iy_i
\]
with $y_i$ words in $x_1,\ldots,x_n$. For each $i$, the map $\mu_2^i$ is given by the composition of the word $y_i^{-1}$ with a triangulation (the canonical obtained without adding any new vertices) of the polygon $r_i$ through $2$-simplices in $\Delta^2_{\lceil \frac l2\rceil}$ (where $l$ is the length of the longest defining relation) and then composing with the word $y_i$. Then $\mu_2|_{(1_G,g_1,g_2)}$ is given by the composition of the $\mu_2^i$.

So $(0,1,\lceil \frac l2\rceil)$ is a homotopical connecting vector for $G$.

\begin{thm}
\label{thm:crit2}
 If $G$ is a group of type $\Ftp m$ and $\chi:G\to \RR$ is a nonzero character, then the following are equivalent:
 \begin{enumerate}
 \item $\chi\in \Sigma^m(G)$;
 \item there exists $(k_0,\ldots, k_m)\in \NN^{m+1}$ such that $\pi_q(\VR{k_q} {G_\chi})$ vanishes in $\pi_q(\VR{k_{q+1}}{G_\chi})$ for each $q=0,\ldots, m-1$;
 \item for every $k\gg 0$ there exists a $G$-equivariant map 
 \[
 \varphi\in \Map({\VR k G}^{(m)},{\VR k G}^{(m)})
 \]
 with $v(\varphi_q\circ\sigma)-v(\sigma)>0$ for every $\sigma\in \Delta^q_k$ and $q=0,\ldots,m$. 
 \end{enumerate}
\end{thm}
\begin{proof}
Condition 1 obviously implies condition 2.

Now, we show that condition 2 implies condition 3. Suppose condition 2. Then each homotopy group $\pi_{q-1}(\VR{k_{q-1}}{G_\chi})$ vanishes in $\pi_{q-1}(\VR{k_q}{G_\chi})$ for $q=1,\ldots,m-1$. We define $\varphi_q$ inductively on the $q$-skeleton with the property 
\[
\varphi_q\in \Map({\VR k G}^{(q)},{\VR{k_q}G}^{(q)})
\]
where the valuation is increased by $\chi(t)$ where $t\in G$ is fixed with $\chi(t)>0$. If $q=0$, then $\varphi_0: \E G^{(0)}\to \E G^{(0)}$ is defined as the mapping $g\mapsto gt$. Then 
\[
\varphi_0\in \Map({\VR k G}^{(0)},{\VR{k_0}G}^{(0)})
\]
and
\[
 v(\varphi_0(g))-v(g)=v(gt)-v(g)=\chi(t)>0.
\]
Now suppose that $\varphi_0,\ldots,\varphi_{q-1}$ have been constructed. If $\sigma:=(1_G,x_1,\ldots,x_q)\in {\VR k G}^{(q)}$, then $\varphi_{q-1}|_{\partial \sigma}$ has image in $\VR{k_{q-1}}{G_{v(\sigma)+\chi(t)}}$. Since $\pi_{q-1}(\VR{k_{q-1}}{G_{v(\sigma)+\chi(t)}})$ vanishes in $\pi_{q-1}(\VR{k_q}{G_{v(\sigma)+\chi(t)}})$, there exists $\phi_\sigma\in \Map(\Delta^q,\VR{k_q}{G_{v(\sigma)+\chi(t)}})$ with boundary $\varphi_{q-1}|_{\partial\sigma}$. If $\tau:=(x_0,\ldots,x_q)\in \E G^{(q)}$, then $\sigma:=(1_G,x_0^{-1}x_1,\ldots,x_0^{-1}x_q)\in (1_G\times G^q)\cap \Delta^q_k$ is a simplex with $\tau=x_0\sigma$. Define
\[
 \varphi_q|_\tau:=x_0\phi_\sigma.
\]
Then $\varphi_q\in \Map({\VR k G}^{(q)},{\VR{k_q}G)}^{(q)}$ with
\[
 v(\varphi_q|_\tau)-v(\tau)=v(\varphi_q|_\sigma)-v(\sigma)\ge v(\sigma)+\chi(t)-v(\sigma)=\chi(t)>0.
\]

 Now, we show that condition 3 implies condition 1. Suppose that there exists a $G$-equivariant map $\varphi$ that increases the valuation $v$. If $\alpha:S^{m-1}\to |\VR k{G_\chi}|$ is a continuous map, then there exists $\tilde \alpha\in \Map(\partial \Delta^m,\VR k{G_\chi})$ homotopic to $\alpha$ by simplicial approximation. Since $G$ is of type $\Ftp m$, there exists $l\ge k$ with $\pi_{m-1}(\VR k G)$ vanishes in $\pi_{m-1}(\VR l G)$. Then there exists $\phi\in \Map(\Delta^m, \VR l G)$ with boundary $\tilde \alpha$. Then $-v(\im\phi)\le n L_\varphi$ for some $n$ where
 \[
  L_\varphi:=\min_{\sigma=(1_G,x_1,\ldots,x_m)\in\Delta^m_l}(v(\varphi(\sigma))-v(\sigma))>0.
 \]
Then $\varphi^{\circ n}\circ\phi\in \Map(\Delta^m,\VR l{G_\chi})$ is a null-homotopy of $\varphi^{\circ n}\circ\tilde\alpha$. It remains to show that $\varphi^{\circ n}\circ\tilde \alpha$ is homotopic to $\tilde \alpha$. By Lemma~\ref{lem:homotopicalhomotopy} there exist $L\ge l$ and a $G$-equivariant homotopy $\eta\in \Map({\VR l G}^{(m)}\times \Delta^1,{\VR L G}^{(m+1)})$ joining $\varphi$ to the inclusion. If $\tau:=((x_0,i_0),\ldots,(x_m,i_m))\in {\VR l{G_\chi}}^{m}\times\Delta^1$ then
\[
 v(\eta|_\tau)\ge v(\eta|_\tau)-v(\tau)\ge \min_{\sigma=((1_G,i_0),(x_1,i_1),\ldots,(x_m,i_m))\in\Delta^m_l\times\Delta^1}(v(\eta|_\sigma)-v(\sigma))=:K.
\]
If $0\le i\le n-1$ then $\eta\circ((\varphi^{\circ i}\circ \tilde\alpha)\times \id_{\Delta^1})\in \Map(\partial \Delta^m\times \Delta^1,\VR L{G_K})$ is a homotopy joining $\varphi^{\circ i}\circ\tilde \alpha$ to $\varphi^{\circ i+1}\circ \tilde \alpha$.
\end{proof}

 \IncMargin{1em}
\begin{algorithm}[t]
 \SetKwInOut{Input}{input}
 \SetKwInOut{Output}{output}
 \Input{A group $G$ of type $\mathrm{F}_2$ with a homotopical connecting vector 
$(n_0,n_1,n_2)$, a non-zero character $\chi:G\to \RR$}
 \Output{``yes, $\chi\in \Sigma^2(G)$''}
 Pick $t\in G$ with $\chi(t)>0$\;
 Define $n:=n_2$\;
 Define function $\varphi_0:\E G^{(0)}\to \E G^{(0)}$ by $g\mapsto gt$\;
 \For{$q\leftarrow 1$ \KwTo $2$}{
 \ForEach{$\sigma:=(1_G,g_1,\ldots,g_q)\in \Delta^q_n$}{
 search for $\phi\in \Map(\Delta^q,\VR n{G_{\chi(t)+v(\sigma)}})$ with $\phi\circ\partial=\varphi_{q-1}\circ\sigma\circ\partial$\;
 define $\varphi_q|_\sigma:=\phi$\;
 }
 }
 \KwRet{``yes''}
 \vspace{0.1cm}
  \caption{$\Sigma^2$-criterion, homotopical version}
  \label{algo:sigma2}
 \end{algorithm}
Both for-loops in Procedure~\ref{algo:sigma2} run over a finite set: $\{1,2\}$ and $\{(1_G,g_1,\ldots,g_q)\in \Delta^q_n\}$ are computable and finite. Also, it is useful if there is a normal form for the elements in $G$, but it is not necessary. We can always choose $n_0=0$ and $n_1=1$. If $0\not=\chi\in \Sigma^2(G)$ then there exists some $n_2\ge 1$ so that the \emph{search} for $\phi$ stops. 
 
\begin{thm}
\label{thm:algo2}
 Let $G$ be a group of type $\mathrm F_2$ with solvable word-problem and homotopical connecting vector $(k_0,k_1,k_2)$ and set $k:=k_2$. If $\chi:G\to \RR$ is a non-zero character, then the following are equivalent:
 \begin{enumerate}
  \item $\chi\in\Sigma^2(G)$;
  \item there exists a $G$-equivariant mapping $\varphi\in\Map({\VR k G}^{(2)},{\VR k G}^{(2)})$  with $v(\varphi(\sigma))-v(\sigma)>0$ for every $\sigma\in \Delta^2_k$;
  \item Procedure~$\ref{algo:sigma2}$ terminates with the output ``yes''.
 \end{enumerate}
\end{thm}
\begin{proof}
We first show that condition 1 implies condition 2. Suppose condition 1 that $\chi\in \Sigma^2(G)$. Theorem~\ref{thm:crit2} implies that there exists a number $l\ge 0$ and for every $n\ge l$ a $G$-equivariant mapping $\varphi_n\in\Map(\VR n G, \VR n G)$ increasing valuation and $\im \varphi^{(2)} \subseteq {\VR l G}^{(2)}$. If $k\ge l$ denote by $\iota_{lk}$ the inclusion $\VR l G \subseteq \VR k G$. Then $\iota_{lk}\circ \varphi_k$ is the required $G$-equivariant mapping. If, conversely, $k<l$, then, since $(k_0,k_1,k_2)$ is a homotopical connecting vector for $G$, Theorem~\ref{thm:muhomotopy} implies that for every $n\ge k$ there exists a G-equivariant mapping $\mu_n\in\Map(\VR n G,\VR n G)$ extending the identity on $\ZZ$ with $\im \mu_n^{(2)}\subseteq {\VR k G}^{(2)}$. Then
\begin{align*}
 \inf_{\sigma\in {\VR l G}^{(2)}} (v(\mu_l(\sigma))-v(\sigma))
 &\ge \min_{\sigma\in \Delta^2_l\cap (1_G\times G^2)}(v(\mu_l(\sigma))-v(\sigma))\\
 &=:K\in \mathbb R
\end{align*}
Then $i\chi(t)\ge -K$ for some $i\in \mathbb N$. Then denote by $\iota_{kl}$ the inclusion $\VR k G \subseteq \VR l G$. Then $\mu_l\circ \varphi_l^{\circ i+1}\circ \iota_{kl}$ is the required $G$-equivariant mapping:
\begin{align*}
 v(\mu_l\circ \varphi_l^{\circ i+1}\circ \iota_{kl}(\sigma))
 &\ge K+v(\varphi_l^{\circ i+1}\circ \iota_{kl}(\sigma))\\
 &\ge K+(i+1)\chi(t)+v(\iota_{kl}(\sigma))\\
 &\ge \chi(t)+v(\sigma)\\
 &>v(\sigma).
\end{align*}

We now assume condition 2 and show condition 1. Let $n\ge k$. Since $(k_0,k_1,k_2)$ is a homotopical connecting vector for $G$, there exists a $G$-equivariant mapping $\mu\in\Map({\VR n G}^{(2)},{\VR k G}^{(2)})$. Condition 2 also provides us with a $G$-equivariant mapping $\varphi\in\Map({\VR k G}^{(2)},{\VR k G}^{(2)})$ with $v(\varphi(\sigma))-v(\sigma)\ge \chi(t)$ for every $\sigma\in {\VR k G}^{(2)}$. Then, as before,
 \[
  v(\mu(x))-v(x)\ge K
 \]
for every $x=(1_G,g_1,g_q)\in\Delta^2_n$ for some $K\in \RR$. Then $i\chi(t)\ge -K$ for some $i\in \NN$. Then $\varphi^{\circ i+1}\circ \mu\in\Map(\VR n G,\VR k G)$ is a $G$-equivariant mapping with
\[
 v(\varphi^{\circ i+1}\circ \mu(\sigma))\ge (i+1)\chi(t)+v(\mu(\sigma))\ge (i+1)\chi(t)+K+v(\sigma)\ge \chi(t)+v(\sigma)
\]
and $(\mbox{im }\varphi^{\circ i+1}\circ \mu )^{(2)}\subseteq {\VR n G}^{(2)}$. Then $\varphi^{i+1}\circ\mu$ satisfies the requirements of Theorem~\ref{thm:crit2}, and therefore $\chi\in \Sigma^2(G)$.
 
 If we assume condition 3, then Procedure~\ref{algo:sigma2} terminates with output ``yes''. The data collected during the computation provide us with a $G$-equivariant mapping that satisfies the requirements of condition 2.

Now, if we assume condition 2, then the mapping sought in Procedure~\ref{algo:sigma2} exists. Since the word problem is solvable in $G$, we can construct the Cayley graph and ultimately also the Vietoris-Rips complex. The search function (which determines $\phi$) then has the same complexity as the shortest path problem (solvable in quadratic time). So Procedure~\ref{algo:sigma2} terminates.
 
\end{proof} 
 
 Now we discuss the coordinates of characters for which $\varphi$ is a proof that the character belongs to $\Sigma^m(G)$. If $\varphi\in \Map(\VR k G,\VR k G)$ is a $G$-equivariant mapping, define
 \begin{align*}
 u_\varphi:\Hom(G,\RR)&\to \RR\\
 \chi&\mapsto \min_{\substack{\bar x=(1_G,g_1,\ldots,g_q)\in 
\Delta^q_k,\\q=0,\ldots,m}}\left(\min_{g\in \varphi_q(\bar x)^{(0)}}\langle 
w(g),u(\chi)\rangle-\min_{g\in \bar x^{(0)}}\langle w(g),u(\chi)\rangle\right).
\end{align*}
Now $u_\varphi$ is a continuous mapping, so 
\[
 U(\varphi):=u_\varphi^{-1}(0,\infty)
\]
is open.

\begin{prop}
 Let $G$ be a group with homotopical connecting vector $(k_0,\ldots,k_m)$. If 
for 
$k:=(k_0,\ldots,k_m)$ there exists a $G$-equivariant mapping $\varphi\in 
\Map(\VR k G,\VR k G)$ then $U(\varphi)\subseteq \Sigma^m(G)$.
\end{prop}
\begin{proof}
 Suppose $\chi\in U(\varphi)$. Then
 \begin{align*}
 0
 &<\min_{\bar x=(1_G,g_1,\ldots,g_q)\in \Delta^q_k,q=0,\ldots,m}\left(\min_{g\in \varphi_q(\bar x)^{(0)}}\langle w(g),u(\chi)\rangle-\min_{g\in \bar x^{(0)}}\langle w(g),u(\chi)\rangle\right)\\
 &=\min_{\bar x=(1_G,g_1,\ldots,g_q)\in \Delta^q_k,q=0,\ldots,m}\left(\min_{g\in \varphi_q(\bar x)^{(0)}}\chi(g)-\min_{g\in \bar x^{(0)}}\chi(g)\right)\\
 &=\min_{\bar x=(1_G,g_1,\ldots,g_q)\in 
\Delta^q_k,q=0,\ldots,m}\left(v(\varphi_q(\bar x))-v(\bar x)\right).
 \end{align*}
 Using a similar argument as in the proof of Theorem~\ref{thm:algo2} we obtain $\chi\in \Sigma^m(G)$.
\end{proof}

 \begin{thm}\cite[2.7~Satz~A]{Renz1988}
 If $G$ is a group of type $\mathrm F_m$, then the subset $\Sigma^m(G)$ is the cone with vertex at $0$ over an open set in $\Hom(G,\RR)$.
\end{thm}
\begin{proof}
 If $\chi\not=0$ is a character on $G$ that belongs to $\Sigma^m$, then there exists some $G$-equivariant mapping $\varphi$ that is a witness to $\chi$ belonging to $\Sigma^m(G)$. Then $U(\varphi)$ is an open neighborhood of $\chi$ in $\Sigma^m(G)$. So $\Sigma^m(G)\setminus\{0\}$ is open. Since for every $\lambda>0$ we have $\lambda\chi\in\Sigma^m(G)$ if and only if $\chi\in\Sigma^m(G)$ and also $0\in \Sigma^m(G)$ if $\Sigma^m(G)\not=\emptyset$ the set $\Sigma^m(G)$ is a cone with vertex at $0$.
\end{proof}

\end{spacing}
\printbibliography

 \bigskip
  \footnotesize
  Elisa Hartmann, \textsc{Faculty of Mathematics, Bielefeld University, D-33615 Bielefeld}\par\nopagebreak
  \textit{E-mail address}, Elisa Hartmann: \texttt{elisa.hartmann@math.uni-bielefeld.de, elisa.hartmann@gmx.net}
\end{document}